\documentclass[12pt, reqno]{amsproc} 

\usepackage{geometry}
\usepackage{color}
\usepackage{verbatim}
\usepackage[utf8x]{inputenc}
\usepackage{t1enc}
\usepackage[english]{babel}

\usepackage{amsmath,amssymb,amsthm}
\usepackage{mathrsfs,esint}
\usepackage{graphicx,wrapfig}
\usepackage[format=hang]{caption}

\newcommand*{\R}{{\mathbb R}}

\newcommand*{\N}{{\mathbb N}}

\newcommand*{\eps}{\varepsilon}

\newcommand*{\pip}{\varphi}

\providecommand*{\vint}[1]{\mathchoice
          {\mathop{\vrule width 5pt height 3 pt depth -2.5pt
                  \kern -9pt \kern 1pt\intop}\nolimits_{\kern -5pt{#1}}}
          {\mathop{\vrule width 5pt height 3 pt depth -2.6pt
                  \kern -6pt \intop}\nolimits_{\kern -3pt{#1}}}
          {\mathop{\vrule width 5pt height 3 pt depth -2.6pt
                  \kern -6pt \intop}\nolimits_{\kern -3pt{#1}}}
          {\mathop{\vrule width 5pt height 3 pt depth -2.6pt
                  \kern -6pt \intop}\nolimits_{\kern -3pt{#1}}}}

\DeclareMathOperator{\diam}{diam}

\DeclareMathOperator{\Mod}{Mod}

\numberwithin{equation}{section}
\theoremstyle{plain}
\newtheorem{thm}[equation]{Theorem}

\newtheorem{prop}[equation]{Proposition}
\newtheorem{cor}[equation]{Corollary}
\newtheorem{lem}[equation]{Lemma}
\newtheorem{sublem}[equation]{Sublemma}

\theoremstyle{definition}

\newtheorem{defn}[equation]{Definition}

\newtheorem{remark}[equation]{Remark}

\begin{document}

\title[On Carrasco Piaggio's theorem -- combinatorial modulus]
{On Carrasco Piaggio's theorem connecting combinatorial modulus and Ahlfors regular conformal dimension} 
\author{Behnam Esmayli}
\address{Department of Mathematical Sciences, P.O.~Box 210025, University of Cincinnati, Cincinnati, OH~45221-0025, U.S.A.}
\email{esmaylbm@ucmail.uc.edu}
\author{Ryan Schardine}
\address{Department of Mathematical Sciences, P.O.~Box 210025, University of Cincinnati, Cincinnati, OH~45221-0025, U.S.A.}
\email{rsch613@gmail.com}
\author{Nageswari Shanmugalingam}
\address{Department of Mathematical Sciences, P.O.~Box 210025, University of Cincinnati, Cincinnati, OH~45221-0025, U.S.A.}
\email{shanmun@uc.edu}
\thanks{
 N.S.~and R.S.~are partially supported by the NSF (U.S.A.) grant DMS~\#2054960. The authors would like to thank Mario Bonk  
 for the reading seminar on the paper of Carrasco Piaggio
 run at MSRI (Berkeley) in Spring 2022 during N.S.'s research stay at MSRI as Chern visiting professor.
 She would also like to thank that august institution for its kind hospitality during her stay.}
\maketitle

\begin{abstract}
The goal of this paper is to provide an expository description of a result of Carrasco Piaggio~\cite{Car}
connecting the Ahlfors regular conformal dimension of a compact uniformly perfect doubling metric space
with the combinatorial $p$-moduli of the metric space. We give detailed construction of a metric associated
with the $p$-modulus of the space when the $p$-modulus is zero, so that the constructed metric is in the
Ahlfors regular conformal gauge of the metric space. To do so, we utilize the tools of hyperbolic filling,
developed first in~\cite{Gro, BP}.
\end{abstract}


\medskip

\noindent
    {\small Mathematics Subject Classification (2020):
Primary: 30L10.
Secondary: 30L05, 51F30, 53C23.
}

\section{Introduction}

Information about the dimensions of a metric space gives us insight into the geometric structure of the space. Given
a metric space, we can consider its topological dimension, Hausdorff dimension, Minkowski dimensions, box-counting
dimensions, or, if the metric space is Ahlfors regular, its Ahlfors regularity dimension.
Here, a metric space is Ahlfors regular if there is a positive real number $Q$ such that the metric space, equipped 
with the $Q$-dimensional Hausdorff measure, sees the measure 
of balls of radius $r>0$ as being comparable to $r^Q$ provided that $r$ is not larger than the diameter of the metric space. 
Note that if the metric space is a Riemannian manifold, then all the
dimensions mentioned above coincide; however, even for subsets of Euclidean spaces, these dimensions might
differ (as, for example, in the case of fractals such as the Sierpinski carpet).
These dimensions, apart from the topological dimension, are not invariant under homeomorphisms.
For example, the Ahlfors regularity dimension of the Euclidean space $\R^n$, equipped with the Euclidean metric $d_{\text{Euc}}$,
is $n$, but when $\R^n$ is equipped with the snowflaked metric, i.e.~the metric given by 
$d(x,y)=d_{\text{Euc}}(x,y)^\eps$ for a fixed $\eps\in(0,1)$, it has Ahlfors regularity dimension $n/\eps$, even
though both metrics generate the same topology on $\R^n$. Note that the identity map of $\R^n$, equipped with the two
metrics, is quasisymmetric, and the Ahlfors regularity dimension of the snowflaked space is larger than that of the Euclidean
space. In the context of Euclidean spaces, each Ahlfors regular metric space that is quasisymmetric to $\R^n$ cannot
have Ahlfors regularity dimension smaller than $n$.

In this note we are interested in a notion called \emph{Ahlfors regular conformal dimension}. This dimension
is the infimum of the
Ahlfors regularity dimensions of all the Ahlfors regular metric spaces that are quasisymmetric to the given metric space.
The Euclidean space $\R^n$ has Ahlfors regular conformal dimension $n$, while Cantor sets, that is, sets that are
totally disconnected and uniformly perfect, have Ahlfors regular conformal dimension $0$. Indeed, it follows
from~\cite[Proposition~15.7]{DS} that if the Ahlfors regular conformal dimension of a metric space is smaller than
$1$, then it must be $0$. We refer the interested reader to~\cite{BK, Kwa, Mur} for a sampling of literature that uses the
notion of Ahlfors regular conformal dimension to explore geometry in various situations.

If the metric space in question is not uniformly perfect, then there are no Ahlfors regular
metric spaces that would be quasisymmetric to it, and in this case the Ahlfors regular conformal dimension is infinite.
Thus, in this note, we focus on compact doubling metric spaces that are uniformly perfect. Uniformly perfect compact
doubling metric spaces are quasisymmetric to some Ahlfors regular metric space
(see~\cite[Theorem~13.3, Corollary~14.15]{Hei} for instance). It follows that in our setting the
Ahlfors regular conformal dimension is finite. The result~\cite[Theorem~1.3]{Car} gives a way of computing this dimension
by the use of discrete approximations of the metric space and the notion of combinatorial modulus. The goal of this
note is to provide an expository proof of this theorem, as the proof found in~\cite{Car} uses a more general construction
that involves many parameters, thus obscuring the basic idea underlying the result. The following is the main focus of this
note, and is Theorem~1.3 from~\cite{Car}.

\begin{thm}\label{thm:main}
Let $(X,d)$ be a compact, doubling, and uniformly perfect metric space. Then the Ahlfors regular conformal dimension
of $(X,d)$ equals the infimum of all $p>0$ for which the combinatorial modulus $M_p(X,d)$ of the metric space $(X,d)$ 
is zero.
\end{thm}

An elegant generalization of the above theorem of Carrasco Piaggio was given by Mathav Murugan~\cite{Mur} in the context
of compact doubling metric spaces that are not uniformly perfect; in this case, the conformal gauge $\mathcal{G}(X,d)$
is not the collection of all Ahlfors regular spaces that are quasisymmetric to $(X,d)$ but the collection of all
metrics $\theta$ on $X$ with respect to which the identity map from $(X,d)$ to $(X,\theta)$ is power quasisymmetric, without insisting on $(X,\theta)$ being Ahlfors regular. A homeomorphism is a power quasisymmetry if the distortion function $\eta(t)$ 
associated with the quasisymmetry is a power function.

With that modification in mind,~\cite{Mur} considers the infimum of the Assouad dimension of $(X,\theta)$ over all $\theta$ in
the gauge rather than the Hausdorff
dimensions. In~\cite{Mur} a doubling measure is first constructed on $(X,d)$ and then an intricate re-distribution of 
the measure is conducted according to the density function $\rho$ on the hyperbolic filling graph in order to 
gain a new metric $\theta_\rho$ that is power quasisymmetric to the original metric $d$. This is a complication that
is necessitated by the lack of uniform perfectness. Our focus is to give a simplified exposition of Carrasco Piaggio's theorem
under the assumption of uniform perfectness, thus avoiding the complications encountered in~\cite{Mur}.

\section{Background notions}

In this section we list the basic notions and standing assumptions used throughout this note, and describe a construction of hyperbolic filling of
compact doubling metric measure spaces.

In this paper, $(X,d)$ denotes a compact metric space. For $x\in X$ and $r>0$, we denote the (open) ball centered at $x$ with 
radius $r$ by $B(x,r):=\{y\in X\, :\, d(x,y)<r\}$. Then, with the center and the radius of a ball $B=B(x,r)$ fixed, and $\lambda>0$,
by $\lambda B$ we mean the concentric ball $B(x,\lambda r)$.

In the rest of the paper, given two quantities $A_1$ and $A_2$, we say that $A_1\lesssim A_2$ if there is a constant $C>0$, dependent only on
the data associated with $X$, such that $A_1\le C\, A_2$. We do this as we do not intend to keep track of the precise constants $C$ apart 
from the parameters it depends on.

\subsection{Basic notions for metric spaces}
\begin{enumerate}
\item[(i)] Recall
that $(X,d)$ is \emph{doubling} if there is a constant $N\ge 1$ such that whenever
$x\in X$, $r>0$, and $A\subset B(x,r)$ is such that 
$A$ is $r/2$-separated, i.e.,~for all $x,y\in A$ with $x\ne y$ we have $d(x,y)\ge r/2$, then the 
cardinality of $A$ is at most $N$. 

\item[(ii)]We will also assume that $(X,d)$ is \emph{uniformly perfect}, that is, 
$X$ has at last two points and there is
a constant $K_d\ge 2$ such that whenever $x\in X$ and $r>0$ with $X\setminus B(x,r)$ 
non-empty, then $B(x,r)\setminus B(x,r/K_d)$ is also non-empty. Without loss of generality we will assume that $K_d>2$.

\item[(iii)] With $(X,d_X)$ satisfying the assumptions outlined in the previous paragraphs, we know from~\cite{Hei} that
there is an Ahlfors regular metric space $(Y,d_Y)$ and a quasisymmetric map $\varphi \colon X\to Y$. Here, by
$(Y,d_Y)$ being \emph{Ahlfors regular} we mean that there is some $Q>0$ such that with $\mu$ the $Q$-dimensional
Hausdorff measure on $Y$ induced by the metric $d_Y$, there is a constant $C_0\ge 1$ such that
\[
\frac{1}{C_0}\, r^Q\le \mu(B_{d_Y}(y,r))\le C_0\, r^Q
\]
whenever $y\in Y$ and $0<r\le \diam(Y)$.  A quasisymmetric map $\varphi \colon X\to Y$ is a homeomorphism such that
we can find a homeomorphism $\eta \colon [0,\infty)\to[0,\infty)$ that satisfies
\[
\tfrac{d_Y(\varphi(x_1),\varphi(x_2))}{d_Y(\varphi(x_1),\varphi(x_3))}
  \le \eta\left(\tfrac{d_X(x_1,x_2)}{d_X(x_1,x_3)}\right)
\]
whenever $x_1,x_2,x_3\in X$ with $x_1\ne x_3$. The quasisymmetry is also called an $\eta$-quasisymmetry when 
we need to specify \emph{the distortion function} $\eta$. 

\item[(iv)] If $(X,d)$ is quasisymmetric to an Ahlfors regular metric space, then by  importing the metric $d_Y$ from $Y$ to $X$,
without loss of generality we may restrict ourselves to consider the class of all metrics $\theta$ on $X$ with respect to which
the natural identity map from $(X,d)$ to $(X,\theta)$ is quasisymmetric and $(X,\theta)$ is Ahlfors regular; the collection $\mathcal{G}(X,d)$
of all such metrics $\theta$ 
is called the \emph{Ahlfors regular quasisymmetric gauge} of the metric space $(X,d)$. For each $\theta\in\mathcal{G}(X,d)$
let $p_\theta$ be the Hausdorff dimension of $(X,\theta)$; note that the $p_\theta$-dimensional Hausdorff measure on $(X,\theta)$
is Ahlfors $p_\theta$-regular.

\item[(v)] The \emph{Ahlfors regular conformal dimension} of $(X,d)$ is the (finite) number
\[
\inf\{p_\theta\, :\, \theta\in\mathcal{G}(X,d)\}.
\]
\end{enumerate}

\subsection{Hyperbolic filling}\label{subsec:hyp-fill}

We now describe a construction of hyperbolic filling of $(X,d)$. Our construction deviates somewhat from that of~\cite{Car},
and follows that of~\cite{BBS, Sh} more closely. The constructions in~\cite{Car, BBS, Sh, BoSa} are variants of the
construction given in~\cite{BP, Gro}, and the book~\cite{BuSch} has a good exposition on this subject.
\begin{enumerate}
\item By rescaling the metric on $X$ if necessary, we may assume without loss of generality that $\diam(X)=\frac12$. 
\item We
fix $\alpha> \max\{2, K_d^3\}$, where $K_d$ is the uniform perfectness constant of $(X,d)$ with $K_d>2$. 
\item Next we fix $\tau>0$ such that 
$\tau>\max\{6, 2(1+\alpha^{-1}), 2K_d^3/(K_d^2-4)\}$. We next fix $x_0\in X$ and set $A_0=\{x_0\}$. 
\item Inductively we construct, for
each positive integer $n$, a set $A_n\subset X$ such that $A_0\subset A_1$, $A_n\subset A_{n+1}$ for each positive integer $n$,
for all $n\in\N$ and all $x,y\in A_n$ with $x\ne y$ we have that $d(x,y)\ge \alpha^{-n}$, and 
\[
\bigcup_{x\in A_n}B(x,\alpha^{-n})=X.
\]
The last condition above tells us that $A_n$ is a \emph{maximal} $\alpha^{-n}$-separated subset of $X$.
\item We set 
\[
V:=\bigcup_{n\in\N \cup \{0\}}A_n\times\{n\}. 
\]
The set $V$ will act as the vertex set for the graph that will be the
hyperbolic filling of $X$. For $k=0,1,\cdots$, we also set 
\[
V_k:=A_k\times\{k\}.
\]
The root vertex, always denoted by $v_0=(x_0,0)$, will play a distinguished role in this note.
\item We say that two vertices $v=(x,n), w=(y,m)\in V$ are neighbors if $(x,n)\ne (y,m)$, $|n-m|\le 1$, and
either $n=m$ and $B(x,\tau \alpha^{-n})$ intersects $B(y,\tau \alpha^{-n})$, or $n\in\{m-1,m+1\}$ and
$B(x,\alpha^{-n})$ intersects $B(y,\alpha^{-m})$. If the two vertices $v$ and $w$ are neighbors, then we denote
this fact by the notation $v\sim w$. Moreover, if $n=m$ we say that $v$ is a horizontal neighbor of $w$, and
if $n\ne m$ then we say that $v$ is a vertical neighbor of $w$. Moreover, when $m=n+1$, we say that
$w$ is a \emph{child} of $v$ and that $v$ is a \emph{parent} of $w$.
\end{enumerate}

\begin{defn}
	The vertex set $V$, together with the above neighborhood relationship, defines a graph $G$, which we will call
	a hyperbolic filling of $X$.
\end{defn}

A vertex $v=(x,n)$ is said to belong to generation $n$; note that $x\in A_n$. 
We also use the notation $|v|$ to denote the generation number of $v$; so $|(x,n)|=n$.
We say that a path $v_1\sim v_2\sim\cdots\sim v_k$
is a vertically descending path if for $j=1,\cdots, k-1$ we have $|v_{j+1}|=|v_j|+1$. If $(x,n)\in V$ and $k\in\N$ such that $k>n$, then by
$D_k((x,n))$ we mean the collection of all vertices $w\in V$ for which there is a vertically descending path from
the vertex $(x,n)$ to $w$ and so that $w$ is of generation $k$. The vertices in $D_k((x,n))$ are said to be the \emph{generation
$(k-n)$ descendants} of the vertex $(x,n)$.

Between each pair of vertices $v,w\in V$ with $v\sim w$, we glue a unit interval to convert $G$ from a combinatorial
graph to a connected metric graph. From~\cite{BBS} we know that such a metric graph is a Gromov hyperbolic geodesic
space, but we will not need the notion of Gromov hyperbolicity in this note. Thus, from now on, $G$ will be the metric
graph. Then, with $\rho \colon V\to (0,\infty)$ fixed, we can use $\rho$ to construct a new metric $d_\rho$ on $G$
as follows. For each $v\in V$ we choose a vertically descending path $v_0\sim v_1\sim\cdots\sim v_k=v$, and set
\[
\pi(v):=\Pi_{j=0}^k\,\rho(v_j).
\]
We then extend $\pi$ linearly to the edges connecting two neighboring vertices $v\sim w$ by setting
$\pi(x)=(1-t)\, \pi(v)+t\, \pi(w)$ where $x$ is at distance $t$ from the vertex $v$ (and hence a distance $1-t$ from
the vertex $w$). Now, for $x,y\in G$ we set
\[
d_\rho(x,y)=\inf_\gamma\int_\gamma \pi\, ds,
\]
where the infimum is over all paths $\gamma$ in $G$ with end points $x$ and $y$. The metric $d_\rho$ has a natural
extension to $X$, as $X$ is identified with the visual boundary of $G$, see for instance~\cite[Theorem~4.2]{Sh}.

\subsection{Combinatorial modulus of $(X,d)$}

In this subsection we fix $p>0$.
For every vertex $v=(x,n)\in V$, there is a naturally associated
ball $B_v:=B(x,\alpha^{-n})$ to $v$. Fix $v =(x,n)\in V$ and $k\in\N$. By $\Gamma_{k}(B_v)$ we mean the collection of all paths 
$(x_1,n+k)\sim (x_2,n+k)\sim\cdots\sim(x_m,n+k)$ in $G$
such that $B(x_1,\alpha^{-(n+k)})$ intersects $B_v$ and $B(x_m,\alpha^{-(n+k)})$ intersects
$X\setminus 2B_v$. Recall from the standard notation that $2B_v=B(x,2\alpha^{-n})$.
A non-negative function $\sigma \colon V\to[0,\infty)$ is said to be 
\emph{admissible} for computing the combinatorial $p$-modulus $\Mod_p(\Gamma_k(B_v))$ if 
\[
\sum_{j=1}^m\sigma((x_j,n+k))\ge 1
\]
whenever the path $(x_1,n+k)\sim (x_2,n+k)\sim\cdots\sim(x_m,n+k)$ is in $\Gamma_k(B_v)$.
We set 
\[
\Mod_p(\Gamma_k(B_v)):=\inf_\sigma\, \sum_{y\in A_{n+k}}\sigma((y,n+k))^p,
\]
with the infimum over all functions $\sigma$ that are admissible for computing $\Mod_p(\Gamma_k(B_v))$.

For $k\in\N$ we set 
\[
\Mod_p(k):=\sup_{v\in V}\Mod_p(\Gamma_k(B_v)).
\]
The \emph{combinatorial $p$-modulus} of $(X,d)$ is the number
\[
M_p(X,d):=\liminf_{k\to\infty}\Mod_p(k).
\]

\begin{remark}\label{rem:setup-n-naught}
Note that if $M_p(X,d)=0$, then for each $\eps_0>0$ we can find a positive integer $n_0$ such that 
$\Mod_p(n_0)<\eps_0$, and so for each $v\in V$ we have 
\[
\Mod_p(\Gamma_{n_0}(B_v))<\eps_0.
\]
There are arbitrarily large choices of $n_0$ satisfying the above condition.

\end{remark}

\begin{remark}
The specific value of $\Mod_p(X,d)$ can possibly depend on the choice of the hyperbolic filling graph,
but the positivity versus nullity of $\Mod_p(X,d)$ is independent of this choice. Such independence is
known to hold as a consequence of Theorem~\ref{thm:main}.
\end{remark}

\subsection{A tale of two graphs}\label{subsec:n-naught}

Unlike in~\cite{Sh}, in the present note we will use not one but two hyperbolic fillings of $(X,d)$. The first was
described in Subsection~\ref{subsec:hyp-fill} above, and the second hyperbolic filling is constructed from the first
by eliminating certain levels of $G$. With $n_0\in \N$ fixed, we now consider 
\[
V[n_0]:=\bigcup_{k\in\N\cup \{0\}}A_{n_0\, k}\times\{k\},
\]
that is, we consider the level $0$ as of level $0$, but level $n_0$ from $V$ is now considered to be level $1$,
level $n_0+n_0$ from $V$  to be level $2$, level $2n_0+n_0$ to be level $3$, and so on. Two vertices
$v=(x,j)$ and $w=(y,k)$ in $V[n_0]$ are declared neighbors if $v\ne w$ and either $j=k$ and the balls 
$B(x,\tau \alpha^{-n_0j})$ and $B(y,\tau\alpha^{-n_0j})$ intersect, or else $j\in \{k-1,k+1\}$ and the balls
$B(x,\alpha^{-n_0j})$ and $B(y,\alpha^{-n_0k})$ intersect. The corresponding graph will be denoted $G[n_0]$.

We emphasize here that the construction of $G[n_0]$ is directly connected to the construction of the original graph $G$,
and both $G$ and $G[n_0]$ are hyperbolic fillings of $X$. On the other hand, if $\Mod_p(n_0)<\eps_0$ for the 
graph $G$, then $\Mod_p(1)<\eps_0$ for the graph $G[n_0]$. 

\emph{The main goal of this note is to give a more expository proof of~\cite[Theorem~1.3]{Car}, which states that 
the Ahlfors regular conformal dimension of $(X,d)$ is equal to the infimum over all $p$ for which $M_p(X,d)=0$.}

\subsection{A characterization of Ahlfors regular conformal gauge in terms of the hyperbolic filling}

In this subsection we describe a result of Carrasco Piaggio characterizing the conformal gauge $\mathcal{G}(X,d)$
in terms of a class of functions on the hyperbolic filling $G$. As our construction of $G$ is not as complicated
as the one found in~\cite{Car}, we refer the interested reader to~\cite[Theorem~1.1, Theorem~6.3, Section~7]{Sh} where 
this theorem~\cite[Theorem~1.1]{Car} is proved for our construction as found in Subsection~\ref{subsec:hyp-fill} above.
We will make use of this theorem in proving the main theorem of this note.

\begin{thm}{{\rm{\cite[Theorem~1.1]{Car}}}} \label{thm:Car1.1}
Let $(X,d)$ be a compact doubling metric space that is also uniformly perfect. Let $G$ be the hyperbolic filling
constructed in Subsection~\ref{subsec:hyp-fill}. Suppose that there is some function $\rho:V\to(0,1)$
such that the following conditions are satisfied:
\begin{enumerate}
\item[(H1)] there are constants $\eta_-,\, \eta_+$ with $0<\eta_-\le \eta_+<1$ such that 
$\rho:V\to[\eta_-,\eta_+]$,
\item[(H2)] there exists a constant $K_0>0$ such that whenever $v,w\in V$ with $v\sim w$, and  
$v_0\sim v_1\sim\cdots\sim v_k=v$ and $v_0\sim w_1\sim\cdots\sim w_n=w$ are vertically descending paths from the root vertex
$v_0$ to $v$ and $w$, respectively, then 
\[
\Pi_{j=0}^k\rho(v_j)\le K_0\, \Pi_{j=0}^n\rho(w_j),
\]
and so in particular, $K_0^{-1}\, \pi(w)\le \pi(v)\le K_0\, \pi(w)$,
\item[(H3)] there is a constant $K_1>0$ such that whenever $v,w\in V$ and 
$v_0\sim v_1\sim\cdots\sim v_k=v$ and $v_0\sim w_1\sim\cdots\sim w_n=w$ are vertically descending paths from the root vertex
$v_0$ to $v$ and $w$, respectively, setting $z_{v,w}$ to be the largest generation vertex that is either common to both paths,
or is in the first path and is a neighbor of a vertex in the second path, or is in the second path and is a neighbor of a vertex in
the first path, we have
\[
\int_\gamma\pi\, ds\ge K_1^{-1}\, \pi(z_{v,w})
\]
for all paths $\gamma$ in $G$ with end points $v,w$, and
\item[(H4)] there exists $p>0$ and a constant $K_2>0$ such that whenever $x\in A_n$ and $k\in\N$ with $k>n$, we have
\[
K_2^{-1}\, \pi((x,n))^p\le \sum_{v\in D_k((x,n))}\pi(v)^p\le K_2\, \pi((x,n))^p.
\]
\end{enumerate}
Then $d_\rho$ belongs to the gauge $\mathcal{G}(X,d)$ and $(X,d_\rho)$ is Ahlfors $p$-regular.
Moreover, 
when $\tau$ also satisfies $\tau>2(1+\alpha^2)$, then
every Ahlfors regular metric $(X,\theta)$
that is quasisymmetric to $(X,d)$ is biLipschitz equivalent to the 
metric $d_\rho$ constructed above,
generated via a function $\rho$
satisfying the above four conditions. 
\end{thm}

In~\cite[Theorem~1.1 and Theorem~6.3]{Sh} the various roles of each of the four conditions (H1)---(H4) are explored, and
we refer the interested reader to~\cite{Sh} for more on this aspect.

\section{Proof of Theorem~\ref{thm:main}: lower bound for the Ahlfors regular conformal dimension}

The focus of this section is to provide a proof of one part of Theorem~\ref{thm:main}, namely that
whenever $p$ is larger than the Ahlfors regular conformal dimension of $(X,d)$, we have that $M_p(X,d)=0$.
In other words, if $p$ is such that $M_p(X,d)>0$, then necessarily the Ahlfors regular conformal dimension of $(X,d)$ is at least $p$.

To this end, we first need the following lemma.

\begin{lem}\label{lem:theta-diam-comparison}
Fix $(x,n)=v\in V_n$, $ n \in \mathbb N \cup \{0\}$.  Let $k\in\N$ and set 
$P_k(v):=\{v'\in V_{n+k}\, :\, B_{v'}\cap\overline{2B_v}\ne\emptyset\}$. For $v'\in P_k(v)$ we set $x'\in X$ such that
$v'=(x',n+k)$.
Suppose that  
$\theta\in\mathcal{G}(X,d)$, with distortion function $\eta$. Then for every
$z\in B_{v'}\setminus \tfrac{1}{K_d}\, B_{v'}$ we have 
\[
\theta(x',z)\le \diam_\theta(B_{v'})\le 2\, \eta(K_d)\, \theta(x',z),
\]
and
\[
\frac{\diam_\theta(B_{v'})}{\diam_\theta(2B_v)}\le 2\, \eta(K_d)\, \eta(3K_d)\, \eta(K_d^2\, \alpha^{-k}).
\]
\end{lem}

\begin{proof}
By uniform perfectness of $X$ and $\diam X = \frac{1}{2}$, we have $B_{v'} \setminus \frac{1}{K_d}B_{v'} \neq \emptyset$. 
Fix an arbitrary point
$z\in B(x',\alpha^{-(n+k)})\setminus B(x',\alpha^{-(n+k)}/K_d)$.
Then 
\[
\theta(x',z)\le \diam_\theta(B_{v'})\le 2\, \sup\{\theta(x', y)\, :\, y\in B_{v'}\}.
\]
Note by the quasisymmetry that for every $y\in B_{v'}=B(x',\alpha^{-(n+k)})$ we have
\[
\frac{\theta(x',y)}{\theta(x',z)}\le \eta\left(\frac{d(x',y)}{d(x',z)}\right)\le \eta\left(\frac{\alpha^{-(n+k)}}{\alpha^{-(n+k)}/K_d}\right)
=\eta(K_d).
\]
Hence 
\[
\theta(x',z)\le \diam_\theta(B_{v'})\le 2\, \eta(K_d)\, \theta(x',z),
\]
which verifies the first claim of the lemma.

Note also that with $w\in 2B_v\setminus \tfrac{1}{K_d}\, B_v$, from the above inequality
we have 
\[
\frac{\diam_\theta(B_{v'})}{\diam_\theta(2B_v)}\le 2\, \eta(K_d)\,  \frac{\theta(x',z)}{\theta(x,w)}.
\]
Two possible cases arise: either $x'\in\tfrac{1}{K_d^2}B_v$, or $x'\not\in\tfrac{1}{K_d^2}B_v$. We consider them separately
to complete the proof.

\noindent {\bf Case 1:} $x'\in\tfrac{1}{K_d^2}B_v$. In this case, note that $d(x',w)\ge d(x,w)-d(x',x)\ge \tfrac{K_d-1}{K_d^2}\alpha^{-n}$,
and so as we have ensured that $K_d>2$ and so $K_d-1>1$, we have
\begin{align*}
\frac{\diam_\theta(B_{v'})}{\diam_\theta(2B_v)}&\le 2\, \eta(K_d)\, \frac{\theta(x',z)}{\theta(x',w)}\, \frac{\theta(x',w)}{\theta(x,w)}\\
  &\le 2\, \eta(K_d)\, 
     \eta\left(\frac{\alpha^{-(n+k)}}{\tfrac{K_d-1}{K_d^2}\alpha^{-n}}\right)\, \eta\left(\frac{3\alpha^{-n}}{\alpha^{-n}/K_d}\right)\\
  &= 2\, \eta(K_d)\, \eta(3K_d)\, \eta(K_d^2\, \alpha^{-k}).
\end{align*}

\noindent {\bf Case 2:} $x'\not\in\tfrac{1}{K_d^2}B_v$. In this case, 
as $v'\in P_k(v)$, we have
\[
d(x,x')\le \alpha^{-(n+k)}+2\alpha^{-n}<3\alpha^{-n}.
\]
So
\begin{align*}
\frac{\diam_\theta(B_{v'})}{\diam_\theta(2B_v)}&\le 2\, \eta(K_d)\, \frac{\theta(x',z)}{\theta(x',x)}\, \frac{\theta(x',x)}{\theta(x,w)}\\
&\le 2\, \eta(K_d)\, \eta\left(\frac{\alpha^{-(n+k)}}{\alpha^{-n}/K_d^2}\right)\, \eta\left(\frac{3\,\alpha^{-n}}{\alpha^{-n}/K_d}\right)\\
&=2\, \eta(K_d)\, \eta(3\, K_d)\, \eta(K_d^2\, \alpha^{-k}).
\end{align*}

The above two cases together complete the proof of the lemma.
\end{proof}

\begin{proof}[Proof of first part of Theorem~\ref{thm:main}]
We will prove that if $p$ is larger than the Ahlfors regular conformal dimension of $(X,d)$, then $M_p(X,d)=0$.

Since Ahlfors regular conformal dimension is defined via an infimum, there exists $q<p$ and $\theta\in\mathcal{G}(X,d)$ such that $(X,\theta)$ is Ahlfors $q$-regular. We fix one such $q$.
Recall from the notation set up in the previous section that for $v=(x,n)\in V$, the ball $B_v$ denotes
$B(x,\alpha^{-n})$; this is a ball with respect to the metric $d$.

Fix $v = (x,n) \in V$, $k\in\N$, and define $\rho':V_{n+k}\rightarrow \mathbf{R}$ by 
\[
\rho'(v') = 
\begin{cases}
    \frac{\diam_{\theta}B_{v'}}{\diam_{\theta}2B_v} & \text{if } B_{v'}\bigcap \overline{2B_v} \neq \emptyset \\
    0 & \text{otherwise.}
\end{cases}
\]
We will now show that $C\rho'$ is admissible for computing $\Mod_p(\Gamma_k(B_v))$ for some constant $C\ge 1$
which is independent of $v$ and $k$. We now wish to estimate $\diam_\theta(2B_v)$. 
This is the focus of the next couple of paragraphs.

Take an arbitrary path $\{v_j\}_{j=1}^N$ in $\Gamma_k(B_v)$ with $v_j = (x_j, n+k)$. 
Each such path induces a "chain" of balls $\{B_{v_j}\}_{j=1}^N$ such that 
$\tau B_{v_j} \bigcap \tau B_{v_{j+1}}\neq \emptyset$ for each $j\in\{1,...,N-1\}$. 
We can assume without loss of generality that every element of such a chain intersects $\overline{2B_v}$, by truncating the 
chain if necessary. We choose a point $w_j\in \tau B_{v_j} \bigcap \tau B_{v_{j+1}}$. Then, by the triangle inequality, 
\[
\theta(x_j,x_{j+1})\leq \theta(x_j,w_j) + \theta(x_{j+1},w_j). 
\] 
Summing up over all $j$ and using the triangle inequality again, we have 
\begin{equation}\label{eq:triangle-sum}
\theta(x_1,x_N)\leq \sum_{j=1}^{N-1}\theta(x_j,x_{j+1}) \leq \sum_{j=1}^{N-1}\theta(x_j,w_j)+\sum_{j=1}^{N-1}\theta(x_{j+1},w_j).
\end{equation}
Recalling the uniform perfectness of $(X,d)$, for every $y_0\in B_{v_j}\setminus B(x_j,\alpha^{-(n+k)}/K_d)$ we see from the 
quasisymmetry between $(X,d)$ and $(X,\theta)$ that
\[
\frac{\theta(x_j,w_j)}{\theta(x_j,y_0)}\le \eta\left(\frac{d(x_j,w_j)}{d(x_j,y_0)}\right)\le \eta(K_d\tau),
\]
and so $\theta(x_j,w_j)\le \eta(K_d\tau)\, \theta(x_j,y_0)\le \eta(K_d\tau)\,\diam_\theta(B_{v_j})$. 
A similar argument with $x_j$ replaced by $x_{j+1}$ and $B_{v_j}$ by $B_{v_{j+1}}$ shows that
$\theta(x_{j+1},w_j)\le \eta(K_d\, \tau)\, \diam_\theta(B_{v_{j+1}})$.
Hence by~\eqref{eq:triangle-sum},
\begin{equation}\label{eq:AAA}
\theta(x_1,x_N)\le 2\, \eta(K_d\tau)\, \sum_{j=1}^N\diam_\theta(B_{v_j}).
\end{equation}
On the other hand, as $B_{v_N}$ intersects $X\setminus B(x,2\alpha^{-n})$, we have that either $x_1=x$,
in which case $d(x_1,x_N)\ge 2\alpha^{-n}-\alpha^{-(n+k)}> \alpha^{-n}$, or else $x_1\ne x$, in which case, by the fact that
$B_{v_1}\cap B_v$ is nonempty, we have by the argument from the case of $x_1=x$ that $d(x,x_N)\ge 2\alpha^{-n}-\alpha^{-(n+k)}$, and so
\[
d(x_1,x_N)\ge [2\alpha^{-n}-\alpha^{-(n+k)}]-[\alpha^{-n}+\alpha^{-(n+k)}]\ge \frac12\, \alpha^{-n};
\] 
hence for each $y\in 2B_v$ we have
\[
\frac{\theta(x_1,y)}{\theta(x_1,x_N)}
\le \eta\left(\frac{d(x_1,y)}{d(x_1,x_N)}\right)\le \eta\left(\frac{4\alpha^{-n}}{\tfrac12\alpha^{-n}}\right)
 =\eta(8).
\]
In the above, we also used the fact that as $B_{v_1}$ intersected $B_v$, necessarily 
$B_{v_1}\subset 2B_v$ and so $d(x_1,y)\le 4\alpha^{-n}$.
It follows that $\theta(x_1,y) \le \eta(8)\, \theta(x_1,x_N)$. 
Therefore by~\eqref{eq:AAA},
\[
\theta(x_1,y)\le 2\, \eta(8)\, \eta(K_d\tau)\, \sum_{j=1}^N\diam_\theta(B_{v_j}).
\]
As $B_{v_1}$ intersects $B_v$, necessarily $x_1\in 2B_v$, and so by triangle inequality, we must have 
\[
\frac12\, \diam_\theta(2B_v) \le \sup_{y\in 2B_v}\theta(x_1,y).
\]
Therefore we have the desired estimate for $\diam_\theta(2B_v)$, namely,
\[
\diam_\theta(2B_v)\le 4\, \eta(8)\, \eta(K_d\tau)\, \sum_{j=1}^N\diam_\theta(B_{v_j}).
\]

From the above-obtained estimate for $\diam_\theta(2B_v)$ it follows that
\[
1\le 4\eta(8)\, \eta(K_d\tau) \, \sum_{j=1}^N\frac{\diam_\theta(B_{v_j})}{\diam_\theta(2B_v)}
=4\eta(8)\, \eta(K_d\tau) \, \sum_{j=1}^N\rho^\prime(v_j).
\]
Hence, setting $\sigma=4\eta(8)\, \eta(K_d\tau)\rho^\prime$ on $V_{n+k}$ and then
extending it by zero to all the other vertices in $G$, we obtain a function that is admissible for computing 
the combinatorial modulus of the family $\Gamma_k(B_v)$.
Now, note that $\sigma$ is supported in $P_k(v):=\{v'\in V_{n+k}\, :\, B_{v'}\cap\overline{2B_v}\ne \emptyset\}$. Therefore
\begin{align*}
\sum_{v'\in V_{n+k}}\sigma(v')^p
&=\sum_{v'\in P_k(v)}\sigma(v')^p\\
&=(4\eta(8)\, \eta(K_d\tau))^p\,\sum_{v'\in P_k(v)}\left(\frac{\diam_\theta(B_{v'})}{\diam_\theta(2B_v)}\right)^p\\
&\le (4\eta(8)\, \eta(K_d\tau))^p\,\sum_{v'\in P_k(v)}\left(\frac{\diam_\theta(B_{v'})}{\diam_\theta(2B_v)}\right)^q\, 
\left(2\eta(K_d)\, \eta(3K_d)\, \eta(K_d^2\, \alpha^{-k})\right)^{p-q},
\end{align*}
where we used Lemma~\ref{lem:theta-diam-comparison} in the last step above.

We need an upper bound for the ratio $\diam_\theta(4B_v)/\diam_\theta(2B_v)$. To do this, we choose $y\in 2B_v\setminus \tfrac{1}{K_d}B_v$, and 
$z\in 4B_v$ such that $\theta(z,x)\ge \tfrac12 \diam_\theta(4B_v)$. Then we have
\[
\frac{\diam_\theta(4B_v)}{\diam_\theta(2B_v)}\le \frac{2\, \theta(z,x)}{\theta(y,x)}\le 2\, \eta\left(\frac{4\alpha^{-n}}{\alpha^{-n}/K_d}\right)
    \le 2\, \eta(4K_d).
\]
Note that if $v'\in P_k(v)$, then 
$B_{v'}\cap \overline{2B_v}$ is non-empty, and so $B_{v'}\subset 4B_v$.
Now by the Ahlfors $q$-regularity of $(X,\theta)$ together with the bounded overlap of
	the balls $B_{v'}$ for $v'\in P_k(v)$ -- with the overlap number independent of $v$ and $k$ -- we have
\[
\sum_{v'\in P_k(v)}\left(\frac{\diam_\theta(B_{v'})}{\diam_\theta(2B_v)}\right)^q
=\frac{\sum_{v'\in P_k(v)}\diam_\theta(B_{v'})^q}{\diam_\theta(4B_v)^q}\, \frac{\diam_\theta(4B_v)^q}{\diam_\theta(2B_v)^q}\lesssim 1,
\]
and so 
\[
\sum_{v'\in V_{n+k}}\sigma(v')^p\lesssim  
\left(\eta(K_d^2\, \alpha^{-k})\right)^{p-q}.
\]
It follows that
\[
\Mod_p(\Gamma_k(B_v))\lesssim \left(\eta(K_d^2\, \alpha^{-k})\right)^{p-q}.
\]
Now, for each $k\in\N$, taking the supremum over all $v\in V$, we see that
\[
M_p(k)\lesssim \left(\eta(K_d^2\, \alpha^{-k})\right)^{p-q}.
\]
As $\lim_{t\to 0^+}\eta(t)=0$, it follows that $\lim_{k\to\infty}M_p(k)=0$, that is, $M_p(X,d)=0$.
\end{proof}

\section{Proof of Theorem~\ref{thm:main}; equality}\label{Sec:4}

In the remainder of this note, we fix $p\in(0,\infty)$ such that $M_p(X,d)=0$. Our goal is to find a metric $\theta$ in the Ahlfors regular quasisymmetric gauge $\mathcal{G}(X,d)$
	of $(X,d)$ such that $(X,\theta)$ is Ahlfors $p$-regular, thus proving that $p$ is an upper bound for the Ahlfors regular conformal dimension of  $(X,d)$ and completing the proof of Theorem~\ref{thm:main}.
	
We now introduce two structural constants.   By the metric doubling property of $(X,d)$, there is a positive integer $N_1=N^6$ 
so that for each $r>0$ and $x\in X$,
	there is at most $N_1$ number of points in any subset of $B(x,6r)$ whose elements are mutually $r/2$-separated.  There is an integer $N_2$ 
such that $N_2$ is an upper bound on the number of horizontal neighbors that each vertex in $V$ has. Indeed, $N_2$ depends only on $\tau$ and the constant associated with the metric doubling of $(X,d)$.
 We fix $\eps >0$ such that
\begin{equation}\label{eq:eps-bound}
	2^{p+2}(N_2+N_1+1)^2\, \eps<1.
\end{equation}
We then fix $\eps_0>0$ so that $N_2^2\eps_0<\eps$, and we choose $n_0$ as in Remark~\ref{rem:setup-n-naught}; 
note that $N_2$ does not depend on the choice of $n_0$, and so this is not a circular condition.

To find such a metric $\theta$, 
we consider the hyperbolic filling graph $G[n_0]$ 
as described in subsection~\ref{subsec:n-naught},  
and a density function $\rho$ on vertex set $V[n_0]$ of $G[n_0]$ that satisfy the hypotheses of Theorem~\ref{thm:Car1.1}. The verification of the hypotheses of Theorem~\ref{thm:Car1.1} will be done once the density function is constructed using a bootstrapping process, see subsection~\ref{Sub:rho}.

Observe that the choice of $n_0 \in \N$ is determined by the choice of $\eps_0 \in (0,1)$, see Remark~\ref{rem:setup-n-naught}, but
we can choose it to be as large as we like. Recall that we require $\tau >6$. We choose $n_0$ to be large enough so that 
\begin{equation}\label{n0-condition}
6+4\alpha^{-n_0}+8\tau\alpha^{-n_0}<\tau<\frac{1}{4}\, \alpha^{n_0}.
\end{equation}
With the above constraints for $n_0$ in mind, 
the graph structure of $G[n_0]$ is commensurate with considering $A_{n_0k}$, $k\in\N$, and with root $v_0$ which is also the root of $G$, so that neighborhood relations are obtained by using the \textit{same value of $\tau$} used in the construction of $G$, but now $\alpha$ is replaced by $\alpha^{n_0}$.

\subsection{An inductive construction of a weight on $G[n_0]$} \label{Sub:G-notation}
Recall from subsection~\ref{subsec:n-naught} that
\begin{equation}\label{eq:V-n0}
V[n_0]=\bigcup_{k\in\{0\}\cup\N} A_{n_0k}\times\{k\}, \text{ with }A_0=\{x_0\}.
\end{equation}
For non-negative integers $k$ we set $V[n_0]_k$ to be the collection of all vertices $v\in A_{n_0k}\times\{k\}$, 
so that $V[n_0]=\bigcup_{k=0}^\infty V[n_0]_k$. Moreover, for each $v\in G[n_0]$ we denote by $x_v$ and $n_v$
the unique elements $x_v\in X$ and $n_v$ a non-negative integer such that $v=(x_v,n_v)$.

We need
the following lemma, which will help us inductively construct the function $\pi_0$ that is intermediary towards
constructing the weight $\pi$ that satisfies the hypotheses of
Theorem~\ref{thm:Car1.1}.

\begin{lem}\label{lem:inductive}
Let $k$ be a positive integer and suppose that there is a constant $K>1$ and two maps
$\pi_0:V[n_0]_k\to(0,\infty)$ and $\pi_1:V[n_0]_{k+1}\to(0,\infty)$ satisfying the following hypotheses:
\begin{itemize}
\item If $w,w'\in V[n_0]_k$ with $w\sim w'$, then 
\begin{equation}\label{eq:Pi-0}
\frac{1}{K}\le \frac{\pi_0(w)}{\pi_0(w')}\le K.
\end{equation}
\item If $v\in V[n_0]_{k+1}$, then there is a vertex $w\in V[n_0]_{k}$ such that $v\sim w$ and
\begin{equation}\label{eq:Pi-01}
1\le \frac{\pi_0(w)}{\pi_1(v)}\le K.
\end{equation}
\end{itemize}
Then there is a map $\pi_0:V[n_0]_{k+1}\to (0,\infty)$ 
such that if $v,v'\in V[n_0]_{k+1}$ with $v\sim v'$, then 
\begin{equation}\label{eq:Pi-0new}
\frac{1}{K}\le \frac{\pi_0(v)}{\pi_0(v')}\le K,
\end{equation}
and moreover, for each $v\in V[n_0]_{k+1}$, we have 
\begin{equation}\label{eq:sum}
\pi_0(v)=\pi_1(v)\ \text{ or, }\ \pi_0(v)=\max\{\pi_1(v')/K\, :\, v'\in V[n_0]_{k+1},\ v'\sim v\}>\pi_1(v).
\end{equation}
Moreover, we also have that if $v\in V[n_0]_{k+1}$, there is some $w\in V[n_0]_{k}$ such that $w\sim v$ and 
\begin{equation}\label{eq:w-parent-child}
1\le \frac{\pi_0(w)}{\pi_0(v)}\le K.
\end{equation}
\end{lem}

\begin{proof} 
Let $v,v'\in V[n_0]_{k+1}$ such that $v\sim v'$.
We first note that if $w,w'\in V[n_0]_k$ are associated with $v,v'\in V[n_0]_{k+1}$ such that $w\sim v$, and $w'\sim v'$
as in the second hypothesis of the lemma, then 
either $w=w'$ or $w\sim w'$ and by~\eqref{eq:Pi-0} and~\eqref{eq:Pi-01},
\begin{equation}\label{eq:nbr-parents}
\frac{1}{K^2}=\frac{1}{K}\, \frac{1}{K}\le \frac{\pi_1(v)}{\pi_1(v')}
=\frac{\pi_1(v)/\pi_0(w)}{\pi_1(v')/\pi_0(w')}\, \frac{\pi_0(w)}{\pi_0(w')}\le \frac{K}{1}\, K=K^2
\end{equation}

At the level $V[n_0]_{k+1}$ there are only finitely many (horizontal) edges; we now impose an orientation on some of  
these edges. If $v\sim v'$ is such an edge so that
$K^{-1}\le \tfrac{\pi_1(v)}{\pi_1(v')}\le K$, then we keep the edge $v\sim v'$ unoriented. If this double inequality fails, then
either $\tfrac{\pi_1(v)}{\pi_1(v')}>K$ or else $\tfrac{\pi_1(v)}{\pi_1(v')}<K^{-1}$, but not both. If 
$\tfrac{\pi_1(v)}{\pi_1(v')}>K$, then we orient the edge $v\sim v'$ so that the orientation is from $v$ to $v'$. If
$\tfrac{\pi_1(v)}{\pi_1(v')}<K^{-1}$, then we orient the edge from $v'$ to $v$.

We now claim that there are no three vertices $v,v',\widehat{v}\in V[n_0]_{k+1}$ with $v'\sim v\sim \widehat{v}$ such that 
$v'\sim v$ is oriented from $v'$ to $v$ and the edge $v\sim \widehat{v}$ is oriented from $v$ to $\widehat{v}$. Indeed, if 
there were such edges, then we must have 
\[
\pi_1(v')>K\, \pi_1(v)>K^2\, \pi_1(\widehat{v}).
\]
In this case, with $w',\widehat{w}\in V[n_0]_k$ associated with $v',\widehat{v}$ satisfying the condition given in~\eqref{eq:Pi-01},
we have that 
\[
\pi_0(\widehat{w})\le K\, \pi_1(\widehat{v})< \frac{\pi_1(v')}{K^2}\le \frac{\pi_0(w')}{K^2},
\]
and so we must have
\[
\frac{\pi_0(\widehat{w})}{\pi_0(w')}<\frac{1}{K^2}.
\]
On the other hand, by~\eqref{n0-condition} we have $1+\alpha^{-n_0}+8\tau\alpha^{-n_0}<\tau$, and $v'\sim v\sim \widehat{v}$, 
and hence we have that $\widehat{w}\sim w'$, and so the above inequality violates the hypothesis~\eqref{eq:Pi-0}. Thus our claim holds true.

From the above argument, it follows that if $v\in V[n_0]_{k+1}$ has a horizontal edge oriented towards $v$, then all oriented
horizontal edges with $v$ as an endpoint must be oriented towards $v$. If $v$ has no horizontal edge directed towards $v$,
then set $\pi_0(v):=\pi_1(v)$. If $v$ belongs to a directed horizontal edge with orientation towards $v$, then we set
\[
\pi_0(v):=\frac{1}{K}\, \max\{\pi_1(v')\, :\, v'\in V[n_0]_{k+1}\text{ with }v'\sim v\}.
\]
With this choice of map $\pi_0:V[n_0]_{k+1}\to(0,\infty)$, we immediately see the validity of the claim~\eqref{eq:sum}; thus we 
only need to verify~\eqref{eq:Pi-0new} when $v,v'\in V[n_0]_{k+1}$ with $v\sim v'$. Let $v$ and $v'$ be such edges. Then
there are three possibilities: {\bf (a):} $\pi_0(v)=\pi_1(v)$ and $\pi_0(v')=\pi_1(v')$, or {\bf (b):} $\pi_0(v)=\pi_1(v)$ and
$\pi_0(v')\ne \pi_1(v')$, or else $\pi_0(v)\ne \pi_1(v)$ and $\pi_0(v')=\pi_1(v')$, and {\bf (c):} $\pi_0(v)\ne \pi_1(v)$ and
$\pi_0(v')\ne \pi_1(v')$.
We address these three cases separately to complete the proof.

\noindent {\bf Case~(a):} In this case, we know that the edge $v\sim v'$ is not a directed edge (for if it is, then at least
one of $\pi_0(v)=\pi_1(v)$ and $\pi_0(v')=\pi_1(v')$ would be false), and so
we have 
\[
\frac{1}{K}\le \frac{\pi_0(v)}{\pi_0(v')}=\frac{\pi_1(v)}{\pi_1(v')}\le  K.
\]

\noindent {\bf Case~(b):} We will treat the case $\pi_0(v)=\pi_1(v)$ and
$\pi_0(v')\ne \pi_1(v')$, the other case following mutatis mutandis. Since $\pi_0(v)=\pi_1(v)$, 
no horizontal edge, with $v$ as a vertex, is directed towards $v$. In particular, the vertex $v\sim v'$ is not directed
towards $v$, and so $K\, \pi_0(v)=K\, \pi_1(v)\ge \pi_1(v')$. It follows that there is a vertex, $\widehat{v'}\in V[n_0]_{k+1}$
(allowing for the possibility that $\widehat{v'}=v$)
with $v'\sim\widehat{v'}$ so that $\pi_0(v')=\pi_1(\widehat{v'})/K$. Therefore
\[
\frac{\pi_0(v)}{\pi_0(v')}=K\, \frac{\pi_1(v)}{\pi_1(\widehat{v'})}
\]
 In this case, with $w,\widehat{w'}\in V[n_0]_k$ such that $v\sim w$ and
$\widehat{v'}\sim \widehat{w'}$, by the assumption~\eqref{n0-condition} on $n_0$, we see that
$w\sim\widehat{w'}$ or $w=\widehat{w'}$, and so from~\eqref{eq:nbr-parents}, we have
\[
\frac{\pi_0(v)}{\pi_0(v')}=\frac{\pi_1(v)}{\pi_1(\widehat{v'})/K}=K\, \frac{\pi_1(v)}{\pi_1(\widehat{v'})}\ge K\, \frac{1}{K^2}=\frac{1}{K}.
\]
On the other hand, by the definition of $\pi_0(v')$, we have that $\pi_1(\widehat{v})\ge \pi_1(v)$. It follows that
\[
\frac{\pi_0(v)}{\pi_0(v')}=K\, \frac{\pi_1(v)}{\pi_1(\widehat{v'})}\le K.
\]
Combining the above two sets of inequalities, in {\bf Case~(b)} also we have that
\[
\frac{1}{K}\le \frac{\pi_0(v)}{\pi_0(v')}\le K.
\]

\noindent {\bf Case~(c):} In this case both $v$ and $v'$ have oriented edges directed towards them; this means that
the edge $v\sim v'$ is not a directed edge. Therefore $K^{-1}\, \pi_1(v')\le \pi_1(v)\le K\, \pi_1(v')$. We can now find
$\widehat{v},\widehat{v'}\in V[n_0]_{k+1}$ and $w,\widehat{w},\widehat{w'}\in V[n_0]_k$ such that
$v\sim \widehat{v}$, $v'\sim \widehat{v'}$, and $w\sim v$,
$\widehat{w}\sim\widehat{v}$, $\widehat{w'}\sim \widehat{v'}$,
and so that $\pi_0(v)=\pi_1(\widehat{v})/K$, $\pi_0(v')=\pi_1(\widehat{v'})/K$. 
By the assumption~\eqref{n0-condition} on $n_0$,
we also have that $w\sim\widehat{w'}$ or $w=\widehat{w'}$, and so by a similar argument to Case~(b) we have that
\[
\frac{\pi_0(v)}{\pi_0(v')}\ge \frac{\pi_1(v)}{\pi_0(v')}=K\, \frac{\pi_1(v)}{\pi_1(\widehat{v'})}\ge K\, \frac{1}{K^2}=\frac{1}{K}.
\]
Reversing the roles of $v$ and $v'$ in the above, by the symmetry between $v$ and $v'$ in this case, we also obtain
\[
\frac{\pi_0(v)}{\pi_0(v')}\le K.
\]

The three cases above complete the proof of the lemma except for~\eqref{eq:w-parent-child}, which we now verify.
Let $v\in V[n_0]_{k+1}$ and $w\in V[n_0]_k$ associated with $v$ as in the condition~\eqref{eq:Pi-01}. 
If $\pi_0(v)=\pi_1(v)$, then
by~\eqref{eq:Pi-01} we have
\[
1\le \frac{\pi_0(w)}{\pi_0(v)}=\frac{\pi_0(w)}{\pi_1(v)}\le K.
\]
If $\pi_0(v)\ne \pi_1(v)$, then $\pi_1(v)\le \pi_0(v)$ and
there is some $v'\in V[n_0]_{k+1}$ such that $\pi_0(v)=\pi_1(v')/K$; in this 
case, let $w'\in V[n_0]_k$ such that $w'\sim v'$ be associated to $v'$ as in~\eqref{eq:Pi-01}; then by~\eqref{n0-condition} 
we know that $w\sim w'$, and so
by~\eqref{eq:Pi-0} and~\eqref{eq:Pi-01} we see that
\[
K\ge \frac{\pi_0(w)}{\pi_1(v)}\ge 
\frac{\pi_0(w)}{\pi_0(v)}=K\, \frac{\pi_0(w)}{\pi_1(v')}
=K\, \frac{\pi_0(w)}{\pi_0(w')}\, \frac{\pi_0(w')}{\pi_1(v')}\ge K\, \frac{1}{K}=1.
\]
This now completes the proof.
\end{proof}

Now we are ready to construct the weight $\pi$ and the associated function $\rho$ on $G[n_0]$.

\subsection{Constructing $\sigma:V[n_0]\to [0,\infty)$, the first candidate for the weight.}\label{Sub:sigma}

Recall from the beginning of this section that for each $v\in G[n_0]$, $x_v$ and $n_v$ denote
the unique element $x_v\in X$ and the unique non-negative integer $n_v$ such that $v=(x_v,n_v)$.

For $v\in V[n_0]$, we set 
$\Gamma[n_0]_1(B_v)$ (with $v=(x_v,n_v)$ and $B_v=B(x_v,\alpha^{-n_0 n_v})$) to be the collection of all paths
$v_1\sim v_2\sim\cdots\sim v_m$ in $G[n_0]$ so that the generation of each $v_i$ with respect to the graph $G[n_0]$
is $n_v+1$ and so that $B_{v_1}\cap B_v\ne \emptyset$ and $B_{v_m}\setminus 2B_v\ne\emptyset$. 
Without loss of generality we restrict our attention to curves that also satisfy
$B_{v_j}\setminus 2B_v\ne\emptyset$ if $2\le j\le m-1$. By the choice of $n_0$, the combinatorial modulus of 
$\Gamma[n_0]_1(B_v)$ satisfies
\[
\Mod_p(\Gamma[n_0]_1(B_v))<\eps_0,
\]
and so there is a function $\sigma_v:V[n_0]\to[0,\infty)$ such that
\begin{enumerate}
\item $\sigma_v(w)=0$ whenever $w\in V[n_0]$ that is of generation different from $n_v+1$ or $\tau B_w$ is not a subset of
$3B_v$ (see below),
\item $\sigma_v(w)=0$ whenever $w\in V[n_0]$ is of generation $n_v+1$ but $\tau B_w\subset B_v$,
\item \label{S1} $\sum_{j=1}^m\sigma_v(v_j)\ge 1$ whenever $v_1\sim v_2\sim\cdots\sim v_m$ is a path belonging to the family 
$\Gamma[n_0]_1(B_v)$,
\item \label{S2} $\sum_{w\in V[n_0]}\sigma_v(w)^p<\eps_0$.
\end{enumerate}
Note that if $w\in V[n_0]$ is of generation $n_v+1$ such that $\tau B_w\subset B_v$, then with $v=(x_v,n_v)$ and $w=(x_w,n_v+1)$,
we have that $d(x_v,x_w)\le \alpha^{-n_0n_v}\, [1-\tau\alpha^{-n_0}]$. Moreover, by the choice of $n_0$ as
in~\eqref{n0-condition}, 
we have that $\tau B_{(x_v,n_v+1)}\subset B_v$. 

Observe that if $(x_w,n_v+1)=w\in V[n_0]$ such that $B_w\cap 2B_v$ is nonempty, then by the largeness of $n_0$ we have
that $\tau B_w\subset 3B_v$.

We now define a function $\sigma:V[n_0]\to[0,\infty)$ by
\[
\sigma(w):=\max\{\sigma_v(w)\, :\, v\in V[n_0]\}.
\]
While, on the face of it, the above maximum appears to be over
an infinite set, given the conditions on $\sigma_v$ listed above we have that
with $w\in V[n_0]$ fixed, the only vertices $v\in V[n_0]$ for which $\sigma_v(w)>0$ are those vertices that are of the previous
generation to the generation of $w$ and so that necessarily $B_w\subset 3B_v$; there are only finitely many such $v$. Indeed,
if $v_1, v_2\in V[n_0]$ are of the previous generation $n_w-1$ such that $B_w\subset 3B_{v_1}\cap 3B_{v_2}$, then 
$d(x_{v_1}, x_w)<3\alpha^{n_0(n_w-1)}$ and $d(x_{v_2}, x_w)<3\alpha^{n_0(n_w-1)}$, and so
$d(x_{v_2},x_{v_1})<6\alpha^{n_0(n_w-1)}$, and thus $v_2\in A_{n_0(n_w-1)}\cap B(x_{v_1},6\alpha^{-n_0(n_w-1)})$, and
there is at most $N_1$ number of such $v_2$. Note that $N_1$ is independent of $n_0$ and $n_w$ and depends solely on
the doubling constant of $(X,d)$. Note that $N_1$ and $N_2$ were defined at the beginning of Section~\ref{Sec:4}.

Then, for each $v\in V[n_0]$ and path $v_1\sim v_2\sim\cdots\sim v_m$ belonging to the family 
$\Gamma[n_0]_1(B_v)$, necessarily $\sum_{j=1}^m\sigma(v_j)\ge \sum_{j=1}^m\sigma_v(v_j)\ge 1$.

\begin{defn}\label{def:Tv}
For each $v\in V[n_0]$, we set $T[n_0](v)$ to be the collection of all vertices 
$(x_w,n_v+1)=w\in V[n_0]$ such that $B_w\subset 6B_v$. 
\end{defn}

We have
\[
\sum_{w\in T[n_0](v)}\, \sigma(w)^p<\,N_2\,\eps_0<\eps.
\]
The factor of $N_2$ is needed above because a vertex $w\in T[n_0](v)$ may show up as part of a path in more than
one family $\Gamma[n_0]_1(B_{v'})$, but all such $v'$ would necessarily be neighbors of $v$ and be of the same generation
as $v$.

While $\sigma$ might sound like a viable candidate to apply Theorem~\ref{thm:Car1.1} to, we do not have good control
over how wildly the value of $\sigma$ might fluctuate between neighboring vertices. For this reason, we need to 
dampen the oscillations of $\sigma$; we do so by two successive procedures, addressed in the next subsection.

\subsection{Improvements from $\sigma$ to $\mu_1$ then to $\mu_2$}\label{Sub:mu2}

Our goal over the next couple of subsections is to 
construct a function $\rho$ that will satisfy properties listed in Theorem~\ref{thm:Car1.1} from the function $\sigma$ defined 
at the end of subsection~\ref{Sub:sigma}. This construction is done in several steps. In this subsection we modify $\sigma$
to obtain another function $\mu_1$, and from $\mu_1$ we obtain the second modification $\mu_2$.

By the fact that $X$ has metric doubling property, there is a positive integer $M$ such that for each vertex $v\in G[n_0]$,
the cardinality of $T[n_0](v)$ is at most $M$.
To see this, let $v\in V[n_0]$ and $v' \in T[n_0](v)$. Since the balls $\tfrac12B_{v'}$ are disjoint and our original space $(X,d)$ is 
metrically doubling, there is some constant $M$ which bounds the number of such balls that might lie inside an 
enlarged ball $3B_v$. This constant necessarily bounds the number of children that $v$ has, since all the children 
of $v$ belong to $T[n_0](v)$.

Now, we take $\eta_- = \left(\frac{\eps}{M}\right)^{1/p}$ and define 
\begin{equation} 
    \mu_1 = (\sigma^p + \eta_-^p)^\frac{1}{p}.
\end{equation}

Defining $\mu_1$ in this way guarantees that $\mu_1 \geq \eta_->0$ and that $\mu_1 \geq \sigma$. 
We have, 
\[
\sum_{v'\in T[n_0](v)}\mu_1(v')^p \leq \sum_{v'\in T[n_0](v)}\sigma^p + \sum_{v'\in T[n_0](v)}\eta_-^p.
\]
Since $\sigma$ satisfies Condition~\eqref{S2} from subsection~\ref{Sub:sigma}, the first term is less than $\eps$. Thus, by
the choice of $\eta_-$ we have
\begin{equation}\label{tag-1}
\sum_{v'\in T[n_0](v)}\mu_1(v')^p<2\eps.
\end{equation}

We now modify $\mu_1$ further to $\mu_2$. To do so, for $(x_w,n_w)=w\in V[n_0]$ we set the \emph{sibling set} $S(w)$ of $w$ to be
\begin{equation}\label{eq:sibling}
S(w):=\{w'\in V[n_0]_{n_w}\, :\, \text{ either }w' \sim w\text{ or }\exists u\in V[n_0]_{n_w}\text{ with }w'\sim u\sim w\}.
\end{equation}
We also set the \emph{immediate sibling set} $SI(w)$ to be 
\begin{equation}\label{eq:immediate-sibling}
SI(w):=\{w'\in V[n_0]_{n_w}\, :\, \text{ either }w' \sim w\text{ or }w'=w\}.
\end{equation}
Note that $SI(w)\subset S(w)$.
Now we set
\begin{equation} \label{eq:mu2}
    \mu_2(v) = 2\max\{\mu_1(v'): v'\in S(v)\},
\end{equation}
When $v\in V[n_0]_k$ at level $n_v=k$, and $\gamma = \{v_i\}_{i=1}^{N}$
is a curve from the family $\Gamma[n_0]_1(B_v)$,
we want to show that 
\begin{equation}\label{eq:H3'-mu1}
\sum_{j=1}^{N-1} \min\bigg\lbrace\mu_2(w)\, :\, w\in V[n_0]_{k+1},\,  w\sim v_j\text{ or }w\sim v_{j+1}\bigg\rbrace\ \ge 1.
\end{equation}
Note that for each $w_0\in V[n_0]_{k+1}$, we have that $\mu_2(w_0)\ge \mu_1(w)\ge \sigma(w)$ whenever 
$w\in V[n_0]_{k+1}$ with $w\sim w_0$, 
it follows that for each $v_j$ in the path $\gamma$, the corresponding minimum for $v_j$ in the above sum is at 
least $\mu_1(v_j)\ge \sigma(v_j)$.
And so by Condition~\eqref{S1} of subsection~\ref{Sub:sigma}, the above inequality follows. (This condition,~\eqref{eq:H3'-mu1},
is referred to in~\cite{Car} as (H$3'$) for $\mu_2$.)

Note also that for $v\in V[n_0]_k$, 
\begin{align*}
\sum_{v'\in T[n_0](v)}\mu_2(v')^p\le 2^p\sum_{v'\in T[n_0](v)}\, \sum_{w\in S(v')}\mu_1(w)^p
  &\le 2^p \sum_{\widehat{v}\in SI(v)}\ \sum_{w\in T[n_0](\widehat{v})}\mu_1(w)^p\\
  &\le 2^p\sum_{\widehat{v}\in SI(v)}\, 2\eps,
\end{align*}
where we used~\eqref{tag-1} in the last inequality. As the number of neighbors of $v$ in $V[n_0]_k$ is at most $N_2$, we have that
\begin{equation}\label{eq:etaPlus}
\sum_{v'\in T[n_0](v)}\mu_2(v')^p\, \le \, 2^{p+1}\eps\, (N_2+1).
\end{equation}
By our choice of $\eps$ in~\eqref{eq:eps-bound}, we have that $2^{p+1}\eps (N_2+1) < 2^{p+2}\eps (N_2+1)^2\ <1$.

While $\mu_2$ has a better behavior in comparison between neighbors, we do not know that the function $\pi$ as
constructed from $\mu_2$ in Theorem~\ref{thm:Car1.1} would satisfy Conditions~(H2--4). For this reason, we need
to modify $\mu_2$ further.

\subsection{Construction of a subgraph $Z[n_0]$}\label{Sub:Z}

Recall that in~\eqref{eq:V-n0} we constructed the graph $G[n_0]$.
Before we proceed further, we construct a subgraph $Z[n_0]$ of $G[n_0]$ by removing some of the vertical edges, as follows.
Considering the same vertex set $V[n_0]$ as for $G[n_0]$, for each vertex $(x_v,n_v)=v\in V[n_0]$ and $w\in V[n_0]_{n_v}$,
we say that the two vertices $v$ and $w$ are neighbors in $Z[n_0]$, denoted $v\sim_Z w$, if $v\sim w$. 
These form the horizontal edges of $Z[n_0]$, that is, all the horizontal edges in $G[n_0]$ are kept
in forming $Z[n_0]$. For $w\in V[n_0]_k$ with $k\ge 1$, we choose exactly one of all the vertices $v\in V[n_0]_{k-1}$ for which
$d(x_w,x_v)\le d(x_w, A_{n_0\, n_v}\setminus\{x_v\})$, and set $v\sim_Z w$. Note that necessarily $v\sim w$ in the graph $G[n_0]$
as well, and that for $x\in A_{n_0k}$, we have that $(x,k)\sim_Z (x,k+1)$.

Observe that $Z[n_0]$ is a subgraph of $G[n_0]$, but with the additional property that if we were to remove all the horizontal edges 
from $Z[n_0]$, then we would obtain a connected tree. 

For $v\in V[n_0]_k$, we define $g(v)_i$, $i \in \{1,...,k\}$ to be the $i-th$ generation 
 ancestor of $v$ in the graph $Z[n_0]$ constructed above. More explicitly, we set $g(v)_0 =v_0$, $g(v)_k = v$, and 
 $g(v)_i$ to be the element $v_i \in V[n_0]_i$ in the unique, strictly vertical path 
 $v_0\sim_Z v_1 \sim_Z \cdots \sim_Z v_{k-1} \sim_Z v$ in $Z[n_0]$ for $i \in \{1,...,k-1\}$.

For each $v\in V[n_0]_k$ we set $D[n_0]_j(v)$ to be the collection of all vertices $w\in V[n_0]_{j+k}$ for which
$v=g(w)_k$. These are the $j$-th descendants of the vertex $v$ in the graph $Z[n_0]$.
Note also that \emph{not all} children of $v$, in the graph $G[n_0]$, are descendants of $v$.

Given a vertex $(x_v,n_v)=v\in V[n_0]$, we call the vertex $(x_v,n_v+1)=w_v$ the \emph{direct descendant} of $v$.  
Note that necessarily $v\sim_Z w_v$. 

\begin{lem}\label{lem:vertical-neigbr-parent}
Suppose that $v,w\in V[n_0]$ with $n_w=n_v+1$ such that $v\sim w$ but $v\not\sim_Z w$. Then $g(w)_{n_v}\sim_Z v$.
\end{lem}

\begin{proof}
Since $v\sim w$, we know that 
\[
d(x_v,x_w)<\alpha^{-n_0\, n_v}+\alpha^{-n_0\, n_w}=\alpha^{-n_0\, n_w}\, (\alpha+1).
\]
Since $g(w)_{n_v}\sim w$, we also have that $d(x_w,x_{g(w)_{n_v}})<\alpha^{-n_0\, n_w}\, (\alpha+1)$. It follows
that
\[
d(x_v,x_{g(w)_{n_v}})< 2\, \alpha^{-n_0\, n_w}\, (\alpha+1)\le \alpha^{-n_0\, n_v}\, (\alpha^{-1}+1)\le \tau\, \alpha^{-n_0\, n_v},
\]
and hence $v\sim g(w)_{n_v}$. As the horizontal edges of $G[n_0]$ also are horizontal edges of $Z[n_0]$, the claim follows.
\end{proof}

\subsection{Improving $\mu_2$ to $\pip$ via $\pi_0$}\label{Sub:phi}

Now, from $\mu_2$ we seek to construct a function $\pi_0:V[n_0]\to(0,\infty)$ by applying Lemma~\ref{lem:inductive}
inductively as follows. First, we set $\pi_0(v_0) = 1$ where $v_0$ is the root of both
$G[n_0]$ as well as $G$, and we set $\pi_0(v) = \mu_2(v)$ for 
$v\in V[n_0]_1$. 
By the definition of $\mu_2$ from~\eqref{eq:mu2} and by~~\eqref{eq:etaPlus} and~\eqref{eq:eps-bound}, we have 
$\eta_-\le \mu_2(v)\le 2\eps^{1/p}<1$ for all $v\in V[n_0]$.
Therefore, if $v,v'\in V[n_0]_1$ such that $v\sim v'$, then
\[
\eta_-\le \frac{\pi_0(v)}{\pi_0(v')}=\frac{\mu_2(v)}{\mu_2(v')}\le \frac{1}{\eta_-}.
\]
Hence the inequality~\eqref{eq:Pi-0} of
Lemma~\ref{lem:inductive} is satisfied. 

Let $n$ be a positive integer and suppose that $\pi_0$ is defined for all the vertices in $V[n_0]_k$ for $k=0,\cdots, n$.
We define the map $\pi_1:V[n_0]_{n+1}\to(0,\infty)$ by setting $\pi_1(w)=\mu_2(w)\, \pi_0(g(w)_n)$ when $w\in V[n_0]_{n+1}$.
Note that $\pi_1$ and $\pi_0$ together satisfy inequality~\eqref{eq:Pi-01} of Lemma~\ref{lem:inductive},  with $K=1/\eta_-$,
because
\[
1\le \frac{\pi_0(g(w)_n)}{\pi_1(w)}=\frac{1}{\mu_2(w)}\le \frac{1}{\eta_-}=K.
\]
By applying Lemma~\ref{lem:inductive} with $k=n$ we now obtain $\pi_0:V[n_0]_{n+1}\to(0,\infty)$ as well.
Thus inductively we have obtained the function $\pi_0:V[n_0]\to(0,\infty)$ satisfying
inequality~\eqref{eq:Pi-0} of Lemma~\ref{lem:inductive} for each positive integer $k$. As a consequence of this lemma,
for $w\in V[n_0]_k$ with $k\ge 1$, we have that either $\pi_0(w)=\pi_0(g(w)_{k-1})\, \mu_2(w)$, or else
\[
\pi_0(w)=\eta_-\, \max\{\pi_0(g(w')_{k-1})\, \mu_2(w')\, :\, w'\in V[n_0]_k\text{ with }w\sim w'\}>\pi_1(w).
\]
We now set $\pip:V[n_0]\to(0,\infty)$ by $\pip(v_0)=1$ and for $w\in V[n_0]\setminus\{v_0\}$, 
\begin{equation}\label{eq:phi}
\pip(w):=\frac{\pi_0(w)}{\pi_0(g(w)_{k-1})}, \ \ \ \text{ when }w\in A[n_0]_k\ \text{ and }k\ge 1.
\end{equation}
Note that $\pip(w)=\mu_2(w)$ if $\pi_0(w)=\pi_1(w)$, and if $\pi_0(w)\ne \pi_1(w)$, then we have that
\[
\pip(w)=\frac{\eta_-}{\pi_0(g(w)_{k-1})}\, \max\{\pi_0(g(w')_{k-1})\, \mu_2(w')\, :\, w'\in V[n_0]_k\text{ with }w\sim w'\}
  >\mu_2(w).
\]
Thus we have $\pip(w)\ge \mu_2(w)$ for all $w\in V[n_0]$. 
Observe that by~\eqref{eq:H3'-mu1}, the function $\pip$ also satisfies~\eqref{eq:H3'-mu1} when $\mu_2$ is
replaced by $\pip$ (recall that in~\cite{Car} this property is called (H$3'$)).

Furthermore, if $w\in V[n_0]_k$ and $\pip(w)\ne \mu_2(w)$, then there is some horizontal neighbor $w'$ of $w$ such that
$\pi_0(w)=\eta_-\, \pi_0(g(w')_{k-1})\, \mu_2(w')$, and then, as $g(w')_{k-1}\sim g(w)_{k-1}$, it follows that
\begin{equation}\label{eq:mu2-pip}
\mu_2(w)\le \pip(w)\le \eta_-\, K\, \mu_2(w')=\mu_2(w')
\le \max\{\mu_2(\widehat{w})\, :\, \widehat{w}\in SI(w)\}.
\end{equation}
In the above, $SI(w)$ denotes the set of immediate siblings of $w$, as defined in~\eqref{eq:immediate-sibling}.
Thus $\eta_-\le \pip(w)$. Moreover, by~\eqref{eq:etaPlus} above, we also have that $\mu_2(w')^p\le 2^{p+1}(N_2+1)\, \eps$
for each $w'\in V[n_0]$, and hence $\mu_2(w')\le 2^{1+1/p}(N_2+1)^{1/p}\, \eps^{1/p}=:\eta_+<1$.
Thus  by \eqref{eq:mu2-pip} we also have that
$\pip(w)\le \eta_+$. In summary, we have that for each $w\in V[n_0]$,
\begin{equation}\label{eq:eta+-gen1}
0<\eta_-\le \pip(w)\le \eta_+<1,
\end{equation}
thus verifying Property~(H1) of Theorem~\ref{thm:Car1.1} for both the graphs $G[n_0]$ and $Z[n_0]$ (both these graphs
have the same vertex set $V[n_0]$).

Observe also that when $w \in V[n_0]_k$, we have
\begin{equation}\label{eq:pi-pip}
\pi_0(w)=\prod_{j=0}^k\pip(g(w)_j).
\end{equation}
Hence, 
if $v,w\in V[n_0]_k$ such that $v\sim w$, then by~\eqref{eq:Pi-0} which is satisfied by $\pi_0$, we have
\[
\prod_{j=0}^k\pip(g(w)_j)\le K\, \prod_{j=0}^k\pip(g(v)_j),
\] 
thus verifying Property~(H2) for the graph $Z[n_0]$.

On the other hand, there is no guarantee that $\pi_0$, which is related to $\pip$ in exactly the same way that
$\pi$ was constructed from $\rho$ in Theorem~\ref{thm:Car1.1} (see~\eqref{eq:pi-pip} above), satisfies
Condition~(H4) of this theorem for the graph $Z[n_0]$. For this reason, we conduct one final modification next.

\subsection{Improving $\pip$ to $\rho$, the final candidate for the weight}\label{Sub:rho}

To obtain validity of Condition~(H4) of Theorem~\ref{thm:Car1.1} for the graph $Z[n_0]$, we have to modify $\pip$
further.

 Recall the definition of $D[n_0]_j(v)$ for each $v\in V[n_0]$ from subsection~\ref{Sub:Z} above. Observe that
$D[n_0]_1(v_0)=V[n_0]_1$ where $v_0$ is the root vertex of both graphs $G[n_0]$ and $Z[n_0]$.
Recall the definition of $SI(w)$ from~\eqref{eq:immediate-sibling}.
Observe by~\eqref{eq:mu2-pip} that for each $v\in V[n_0]_k$, 
\begin{align*}
\sum_{w\in D[n_0]_1(v)}\pip(w)^p
&\le \sum_{w\in D[n_0]_1(v)}\max\{\mu_2(\widehat{w})^p\, :\, \widehat{w}\in SI(w)\}\\
&\le \sum_{w\in D[n_0]_1(v)}\ \sum_{w'\in SI(w)}\ \mu_2(w')^p.
\end{align*}
If $w',w\in V[n_0]_k$ with $k\ge 1$ and $w'\in SI(w)$, then by Lemma~\ref{lem:vertical-neigbr-parent} we 
know that $g(w')_{k-1}\sim g(w)_{k-1}$, and so by~\eqref{eq:etaPlus} and the fact that
$D[n_0]_1(\widehat{v})\subset T[n_0](\widehat{v})$ whenever $\widehat{v}\in V[n_0]$, we have
\begin{align*}
\sum_{w\in D[n_0]_1(v)}\pip(w)^p
&\le \sum_{v'\in SI(v)}\ \ \sum_{w'\in D[n_0]_1(v')}\mu_2(w')^p\\
&\le \sum_{v'\in SI(v)}\ \ 2^{p+1}(N_2+1)\, \eps \le 2^{p+1} (N_2+1)^2\, \eps<1,
\end{align*}
where we used the condition on $\eps$ from~\eqref{eq:eps-bound} from the beginning of the present section in concluding the
last inequality. Note also that the direct descendant of $v=(x_v,n_v)$ is the unique vertex $w_v\in V[n_0]_{n_v+1}$ such that
$w_v=(x_v,n_v+1)$; so we can choose $\omega(v)>1$ such that
\begin{equation}\label{eq:defn-omega}
\omega(v)^p\, \pip(w_v)^p+\sum_{w\in D[n_0]_1(v)\setminus w_v}\pip(w)^p=1.
\end{equation}

We now define $\rho:V[n_0]\to(0,\infty)$ as follows. Having chosen $\rho(v_0)=1$, we note that there is exactly
one vertex $w_{v_0}\in V[n_0]_1$ such that $w_{v_0}=(x_0,1)$, where the root vertex is  $v_0=(x_0,0)$. 
Note also that $D[n_0]_1(v_0)=V[n_0]_1$. For $w\in V[n_0]_1$, we set
$\rho(w)=\pip(w)$ if $w\ne w_{v_0}$, and $\rho(w_{v_0})=\omega(w_0)\pip(w_{v_0})$.
Then for each $v\in V[n_0]_k$, $k\in\N$, we set
\[
\rho(w)=\begin{cases} \pip(w)&\text{ if }w\ne w_v,\\
                                   \omega(v)\pip(w)&\text{ if }w=w_v.\end{cases}
\]
This defines $\rho:V[n_0]\to\R$. We wish to apply Theorem~\ref{thm:Car1.1} to the hyperbolic filling graph $G[n_0]$
together with the density $\rho$ given above, to conclude that the Ahlfors regular conformal dimension of $X$ is at most $p$.
That is the focus of the rest of the discussion in this paper. The conditions~(H1), (H2), and (H4) are more readily verified,
and we do so next. Condition~(H3) is more complicated, and we postpone its proof to the next subsection.

\begin{enumerate} 
\item We now verify Condition~(H1) of Theorem~\ref{thm:Car1.1} for our choice of $\rho$.
 As $\eta_-\le \mu_2\le \rho$, it follows that for each $w\in V[n_0]$ we have that
$\rho(w)\ge \eta_-$. Furthermore, as $X$ is uniformly perfect, by the choice of $\alpha$ we know that for each
$v\in V[n_0]$ the descendant set $D[n_0]_1(v)$ has at least two points. Indeed, it contains $w_v=(x_v,n_v+1)$,
as well as at least one point $w'=(x_{w'},n_v+1)$ with $x_{w'}\in B(x_v,\alpha^{-n_0\, (n_v+1)})\setminus B(x_v, \alpha^{-n_0\, (n_v+2)})$.
Therefore, we have from the equation~\eqref{eq:defn-omega} that defines the value of $\omega(v)$, that
\[
\omega(v)^p\pip(w_v)^p\le 1-\pip(w')^p\le 1-\eta_-^p,
\]
and so
\begin{equation}\label{eq:omega-estimates}
\eta_-\le \omega(v)\pip(w_v)=\rho(w_v)\le (1-\eta_-^p)^{1/p}<1.
\end{equation}
From this, and also by applying~\eqref{eq:eta+-gen1} to the case $w\ne w_v$, we see that for each $w\in V[n_0]$,
\begin{equation}\label{eq:rho-H1}
\eta_-\le \rho(w)\le \max\{\eta_+,(1-\eta_-^p)^{1/p}\}<1.
\end{equation}
Hence $\rho$ satisfies Condition~(H1) of Theorem~\ref{thm:Car1.1} for both $G[n_0]$ as well as $Z[n_0]$.

\item We now set $\pi:V[n_0]\to (0,\infty)$ by
\[
\pi(w)=\prod_{j=1}^{n_w}\rho(g(w)_j),
\]
where the ancestry $g(w)$ of a vertex $w\in V[n_0]_{n_w}$ is set out in subsection~\ref{Sub:Z} above.
Now we verify Condition~(H2) of Theorem~\ref{thm:Car1.1} for the graph $Z[n_0]$.

Let $v,w\in V[n_0]$ such that $v\sim_Z w$. If $v$ and $w$ are in different generations of $G[n_0]$,
or equivalently, in different generations of $Z[n_0]$, then  we have that
$\tfrac{\pi(w)}{\pi(v)}=\rho(w)$ when $n_w=n_v+1$ and $\tfrac{\pi(w)}{\pi(v)}=\tfrac{1}{\rho(v)}$ if $n_v=n_w+1$,
and so by~\eqref{eq:rho-H1} We have
\[
\frac{1}{K}\le \frac{\pi(w)}{\pi(v)}\le K,\ \ \ K=1/\eta_-.
\]
If $n_v=n_w=:m$, then consider the unique vertical path $v_0\sim_Z w_1\sim_Z\cdots\sim_Z w_m=w$
and the unique vertical path $v_0\sim_Z v_1\sim_Z\cdots\sim_Z v_m=v$, and let $k\in\{0, 1, \cdots, m\}$
be the largest integer for which $w_k=v_k$. As $v\sim w$ and so $v\ne w$, necessarily $0\le k<m$. 
Moreover, we also have that $v_j\sim_Z w_j$ or $v_j=w_j$ when $j=1,\cdots, m$.

We now show that for $k+1 \le j <m $, neither $w_{j+1}$ nor $v_{j+1}$ is a direct descendant of an earlier vertex. 
To do so, first 
note that we cannot have both $w_{k+1}$ and $v_{k+1}$ be direct descendants of $v_k$, by the maximality of $k$. 
If there is some
$j\in\{k+1,\cdots, m-1\}$ such that $w_{j+1}$ is a direct descendant of $w_j$ and $v_{j+1}$ is a direct
descendant of $v_j$, then we have that $d(x_{v_j}, x_{w_j})\ge \alpha^{-n_0j}$ and at the same time
$d(x_{v_j}, x_{w_j})<2\tau \alpha^{-n_0(j+1)}$, from which we would have that $\alpha^{n_0}\le 2\tau$,
which is not possible because of~\eqref{n0-condition}. 

Now, for any $j\in\{k+1,\cdots, m-1\}$, suppose that $v_{j+1}$ is the direct descendant
of $v_j$; then $w_{j+1}$ is not the direct descendant of $w_j$. In this case, we have that
$d(x_{v_j},x_{w_{j+1}})=d(x_{v_{j+1}},x_{w_{j+1}})<2\tau\alpha^{-n_0(j+1)}$,
while $d(x_{v_j},x_{w_j})\ge \alpha^{-n_0j}$. Then
\begin{align*}
d(x_{w_{j+1}},x_{w_j})\ge d(x_{v_j},x_{w_j})-d(x_{v_j},x_{w_{j+1}})&> \alpha^{-n_0j}-2\tau\alpha^{-n_0(j+1)}\\
&=\alpha^{-n_0j}\, [1-2\tau\alpha^{-n_0}],
\end{align*}
and as by~\eqref{n0-condition} we have that
\[
d(x_{v_j},x_{w_{j+1}})<2\tau\alpha^{-n_0(j+1)} <\alpha^{-n_0j}\, [1-2\tau\alpha^{-n_0}]< d(x_{w_{j+1}},x_{w_j}),
\]
it would follow that $w_j$ cannot be a vertical neighbor of $w_{j+1}$ in $Z[n_0]$ by the construction of $Z[n_0]$ 
from subsection~\ref{Sub:Z}. It follows that for $j\ge k+1$, neither $w_{j+1}$ nor $v_{j+1}$ are direct descendants of
$w_{j}$, $v_{j}$. 

Thus, at most one of $w_{k+1}$, $v_{k+1}$ is the direct descendant, and for $j\ge k+2$
neither of $w_j,v_j$ is a direct descendant. Hence, in the situation that $n_v=n_w$, we have that
\[
\frac{\pi(v)}{\pi(w)}=\lambda_{vw}\, \frac{\pi_0(v)}{\pi_0(w)}, 
\]
with $\lambda_{vw}=1$ if neither $v_{k+1}$ nor $w_{k+1}$ is a direct descendant, and
$\lambda_{vw}=\omega(v_k)$ if $v_{k+1}$ is the direct descendant of $v_k$, and
$\lambda_{vw}=1/\omega(v_k)$ if $w_{k+1}$ is the direct descendant of $v_k$. 
Note that by the construction of $\pi_0$, we know the validity of comparison~\eqref{eq:Pi-0} with $K=1/\eta_-$.
Thus, to verify~(H2) we only need
estimates for $\omega(v_k)$ in the latter two cases. 

Without loss of generality we assume that 
$v_{k+1}$ is the direct descendant of $v_k$.
Note by~\eqref{eq:omega-estimates} that $\eta_-\le \omega(v_k)\pip(v_{k+1})\le (1-\eta_-^p)^{1/p}$, and so
by~\eqref{eq:eta+-gen1} we have
\[
\frac{\eta_-}{\eta_+}\le 
\frac{\eta_-}{\pip(v_{k+1})}\le \omega(v_k)\le \frac{(1-\eta_-^p)^{1/p}}{\pip(v_{k+1})}\le \frac{(1-\eta_-^p)^{1/p}}{\eta_-}
  =(\eta_-^{-p}-1)^{1/p}.
\]
Hence we have
\[
\frac{\eta_-}{K\, \eta_+}\le 
\frac{\pi(v)}{\pi(w)}\le K\, (\eta_-^{-p}-1)^{1/p},
\]
and so Condition~(H2) is verified for the graph $Z[n_0]$ with constant 
\[
K_0=\max\{\eta_-^{-1},\, \eta_-^{-1}\, (\eta_-^{-p}-1)^{1/p},\, \eta_+/\eta_-^2\}.
\]
Finally we verify Condition~(H2) for the full graph $G[n_0]$. Notice that the only edges in $G[n_0]$ that are
absent in $Z[n_0]$ are some of the vertical edges. If $v\sim w$ is such an edge, then 
without loss of generality, we can assume that $n_w=n_v+1$; then
we know
from Lemma~\ref{lem:vertical-neigbr-parent} that $v\sim_Z g(w)_{n_v}$, and so comparing $\pi(v)$ to 
$\pi(g(w)_{n_v})$ and then $\pi(g(w)_{n_v})$ to $\pi(v)$, we obtain Condition~(H2) for the graph $G[n_0]$
with constant
\begin{equation}\label{eq:K0}
K_0=\left(\max\{\eta_-^{-1},\, \eta_-^{-1}\, (\eta_-^{-p}-1)^{1/p},\, \eta_+/\eta_-^2\}\right)^2.
\end{equation}
Note that when two vertices $v_1,v_2\in G[n_0]$ are horizontal neighbors, they are then necessarily neighbors in
$Z[n_0]$ as well, and so we have $K_0^{-1/2}\, \pi(v_1)\le \pi(v_2)\le K_0^{1/2}\, \pi(v_1)$.

\item We now verify Condition~(H4) of Theorem~\ref{thm:Car1.1}. Indeed, for vertices $w\in V[n_0]$,
by the definition of $\omega(w)$ from~\eqref{eq:defn-omega} we have that
\[
\sum_{u\in D[n_0]_1(w)}\rho(u)^p=1,
\]
and so 
\[
\sum_{u\in D[n_0]_1(w)}\pi(u)^p=\pi(w)^p\, \sum_{u\in D[n_0]_1(w)}\rho(u)^p=\pi(w)^p.
\]
Thus, induction now gives the validity of Condition~(H4) for the graph $Z[n_0]$ with constant $1$. Given a vertex $v\in V[n_0]_n$
for some $n\in\N$, for each $k>n$ by $DG[n_0]_k(v)$ we mean the collection of all vertices $w$ in $V[n_0]_k$
such that there is a vertically descending path in $G[n_0]$ from $v$ to $w$. Note that $w\in D[n_0]_k(v)\subset DG[n_0]_k(v)$.
For $w\in DG[n_0]_k(v)\setminus D[n_0]_k(v)$, we know that $g(w)_n\sim_Z v$ and so $g(w)_n\sim v$.
Therefore we also have that
\[
DG[n_0]_k(v)\subset\bigcup_{v'\in V[n_0]_n,\, v'\sim_Z v\text{ or }v=v'}\ D[n_0]_k(v').
\]
Hence by Condition~(H2) we have
\begin{align*}
\pi(v)^p= \sum_{w\in DG[n_0]_k(v)}\, \pi(w)^p&\le\sum_{v'\in SI(v)}\ \sum_{w\in D[n_0]_k(v')}\pi(w)^p\\
  &=\sum_{v'\in SI(v)}\pi(v')^p\\
  &\le K_0^p\ \#~SI(v)\ \pi(v)^p\\
  &\le (N_2+1)\, K_0^p\, \pi(v)^p.
\end{align*}
Thus Condition~(H4) is verified for the graph $G[n_0]$ with constant 
\[
K_2=(N_2+1)\, K_0^p,
\]
where $N_2$ is the maximal number of horizontal neighbors of vertices in $V[n_0]$ (which depends solely on the 
doubling constant and the choice of $\tau$ but is independent of $n_0$).
\end{enumerate}

\subsection{Verifying~(H3) for the weight $\rho$ on $Z[n_0]$.}\label{Sub:H3}

Note that as $\rho\ge\pip\ge \mu_2$, by~\eqref{eq:H3'-mu1} we know that~\eqref{eq:H3'-mu1} is also satisfied by $\rho$, that is,
\begin{equation}\label{eq:H3'-rho}
\sum_{i=1}^{N-1}\min\{\rho(w)\, :\, w\in V[n_0]_{k+1}, w\sim v_i\text{ or }w\sim v_{i+1}\}\ge 1
\end{equation}
whenever $v\in V[n_0]$ with $n_v=k$ and $\{v_i\}_{i=1}^N$ is a curve in $\Gamma[n_0]_1(B_v)$.
Recall the definition on immediate sibling set $SI$ from~\eqref{eq:immediate-sibling}.

Condition~(H3) concerns bounding integral of $\pi$ along paths from below. To verify Condition~(H3) it is 
convenient to introduce a new length that we can manipulate more easily. Towards the end of the section we 
will verify comparability of this length to the path integrals of $\pi$. With the function $\pi$ in hand, set
\begin{equation*}
\pi^*(v):=\min\{\pi(v')\, :\, v'\in SI(v)\},\ \ \ \rho^*(v):=\min\{\rho(v')\, :\, v'\in SI(v)\}
\end{equation*}
for $v\in V[n_0]$. Then by~\eqref{eq:H3'-rho} we know that  if $v\in V[n_0]$ and $\{v_i\}_{i=1}^N$ is a path in
$\Gamma[n_0]_1(B_v)$, 
\begin{equation}\label{eq:rho-star-H3'}
\sum_{i=1}^{N-1}\min\{\rho^*(v_i),\, \rho^*(v_{i+1})\}\ge 1.
\end{equation}
If $e$ is an edge in $Z[n_0]$, that is, if $e=\, v\sim_Z v'$, then set
\begin{equation}\label{eq:def-ell}
\ell(e)=\begin{cases}\min\{\pi^*(v),\, \pi^*(v')\} &\text{ if }n_{v'}=n_v,\\
   K_0^2\, \pi^*(v') &\text{ if }n_{v'}=n_v+1.\end{cases}
\end{equation}
The $\ell$-length of a path is the sum of the $\ell$-lengths of each edge that constitutes the path.

We proceed to find estimates for the $\ell$-length for various types of paths. Given the variety of paths we can have, 
there will be multiple lemmas and sublemmas, but the proofs are based only on a few basic observations.

We have already proven that $\pi$ satisfies Condition~(H2) for both $Z[n_0]$ and $G[n_0]$.
Therefore, for $v\in V[n_0]$ with $n_v\ge 1$, we have that 
\[
\pi^*(g(v)_{n_v-1})\le \pi(g(v)_{n_v-1})\le K_0\, \pi(v)\le K_0^2\, \pi^*(v).
\]
Let $v'\in SI(g(v)_{n_v-1})$. Then from the above set of inequalities and the definition of $\ell$ we see that
\begin{equation}\label{eq:vert-horiz}
\ell(v'\sim_Z g(v)_{n_v-1})\le \ell(v\sim_Z g(v)_{n_v-1}).
\end{equation}

\begin{lem}\label{lem:2.10}
If $\{v_i\}_{i=1}^N\in \Gamma[n_0]_1(B_v)$ for some $v\in V[n_0]$,
then 
\[
\ell(\{v_i\}_{i=1}^N)=\sum_{i=1}^{N-1}\ell(v_i\sim_Z v_{i+1})\ge \max\{\pi^*(v),\pi^*(g(v_1)_{n_v})\}.
\]
\end{lem}

\begin{proof}
For $i=1,\cdots, N$ let $\widehat{v_i}\in SI(v_i)$ such that $\pi^*(v_i)=\pi(\widehat{v_i})$. Setting $n:=n_v$, we then
see that
\begin{align*}
d(x_{g(\widehat{v_i})_{n}},x_{\widehat{v_i}})&<\alpha^{-n_0\, n}+\alpha^{-n_0\, (n+1)},\\
d(x_{\widehat{v_i}},x_{v_i})&<2\tau \alpha^{-n_0\, (n+1)},\\
d(x_{v_i},x_v)&<2\alpha^{-n_0\, n}+\alpha^{-n_0\, (n+1)},
\end{align*}
and combining these, we see that 
\[
d(x_{g(\widehat{v_i})_{n}},x_v)<\alpha^{-n_0\, n}\, [3+2(1+\tau)\alpha^{-n_0}]<\tau\, \alpha^{-n_0\, n},
\]
where we used~\eqref{n0-condition} to obtain the last inequality above. Now it follows for each $i=1,\cdots, N$,
that $g(\widehat{v_i})_n\sim_Z v$. 
From the above and from the largeness requirement on $n_0$ given in~\eqref{n0-condition}, we also get that
\[
d(x_{g(\widehat{v_i})_{n}}, x_{g(\widehat{v_1})_{n}})<\tau\, \alpha^{-n_0\, n},
\]
from which we get that $g(\widehat{v_i})_n\sim_Z g(\widehat{v_1})_n$.
Finally, as $g(\widehat{v_i})\sim v$ for each $i$, we have that 
\begin{align*}
\ell(\{v_i\}_{i=1}^N)=\sum_{i=1}^{N-1}\min\{\pi(\widehat{v_i}),\pi(\widehat{v_{i+1}})\}
&\ge\sum_{i=1}^{N-1}\pi^*(v)\, \min\{\rho(\widehat{v_i}), \rho(\widehat{v_{i+1}})\}\\
  &\ge \pi^*(v)\sum_{j=1}^{N-1}\min\{\rho^*(v_i),\rho^*(v_{i+1})\}\\
   &\ge \pi^*(v),
\end{align*}
where we have used inequality~\eqref{eq:rho-star-H3'} in the last step. 
A similar argument with $g(v_1)_n$ playing the role of $v$
in the second to last step above also gives
\[
\ell(\{v_i\}_{i=1}^N)\ge \pi^*(g(v_1)_n),
\]
and combining the above two sets of inequalities yields the claim of the lemma.
\end{proof}

	\begin{remark}\label{rem:simples}
	Let $v,v'\in V[n_0]$. We want to show that the length of every curve $\gamma$ in $Z[n_0]$ with end points $v,v'$ is controlled
	from below by a fixed constant multiple of $\pi(z_{v,v'})$,  where $z_{v,v'}\in V[n_0]$ is the 
	vertex as described in Condition~(H3) of Theorem~\ref{thm:Car1.1}.
	We first deal with the easy situations.
	If either of the following cases occurs, then $z_{v,v'}\in\{v,v',g(v)_{k-1}, g(v')_{k-1}\}$, and the claim would follow:
\begin{itemize}
	\item $v,v'\in V[n_0]_k$ and $g(v)_{k-1}=g(v')_{k-1}$ or $g(v)_{k-1}\sim_Zg(v')_{k-1}$.
	\item $v\in V[n_0]_k$ and $g(v)_{k-1}\sim_Z v'$ or $g(v)_{k-1}=v'$.
	\item $v'\in V[n_0]_k$ and $g(v')_{k-1}\sim_Z v$ or $g(v')_{k-1}=v$.
\end{itemize}
Hence, from now on, we will assume that $v,v'\in V[n_0]$ do not satisfy any of the above three cases.
\end{remark}

Within the next few lammata, culminating in Lemma~\ref{lift6}, we show that any path $\gamma$ in $Z[n_0]$ that joins two vertices 
$v,v'\in V[n_0]_k$ can be replaced by a shorter path (in $\ell$-length) connecting $v$ to $v'$ that remains within the first $k$ levels.

\begin{sublem}\label{lem:short}
	If $v,v'\in V[n_0]_k$ and that $x_{v'}\in 2B_{g(v)_{k-1}}$, then $g(v)_{k-1}\sim_Zg(v')_{k-1}$.
\end{sublem}

\begin{proof}
	By triangle inequality, we have that 
	\begin{align*}
		d(x_{g(v)_{k-1}}, x_{g(v')_{k-1}})&\le d(x_{g(v)_{k-1}}, x_v)+d(x_v,x_{v'})+d(x_{v'},  x_{g(v')_{k-1}}))\\
		&<2[\alpha^{-n_0k}+\alpha^{-n_0(k-1)}]+2\alpha^{-n_0(k-1)}\\
		&=2(\alpha^{-n_0}+2)\, \alpha^{-n_0(k-1)}<\tau\, \alpha^{-n_0(k-1)},
	\end{align*}
	where we used the condition on $n_0$ from~\eqref{n0-condition}. 
\end{proof}
	
	\begin{sublem}\label{lift1}
		Let $v \in V[n_0]_k$ and suppose $w',w \in V[n_0]_{k+1}$, $w'\sim_Z w$, $B_{w'} \cap 2B_v \neq \emptyset$, $B_{w} \cap (X\setminus 2B_v) \neq \emptyset$. Then $v \sim_Z g(w)_k$. 		
	\end{sublem}

	\begin{proof}	
		Observe that 
		\begin{align*}			
			d(x_v,x_w) \leq d(x_v,x_{w'}) + d(x_{w'},x_w) \leq 2\alpha^{-n_0k}+\alpha^{-n_0(k+1)} + 2\tau \alpha^{-n_0(k+1)}.		
		\end{align*}
		
		On the other hand,		
		\begin{equation}\label{eq:child-parent}
		d(x_{g(w)_k},x_w) \leq \alpha^{-n_0k} + \alpha^{-n_0(k+1)}.		
		\end{equation}
		
		Combining the two inequalities, and by choice of $n_0$, we obtain		
\[
		d(x_v,x_{g(w)_k}) < \tau \alpha^{-n_0k}.
\]		
Moreover, as $B_w\setminus 2B_v$ is nonempty, $d(x_w,x_v)\ge 2\alpha^{-n_0k}-\alpha^{-n_0(k+1)}$, and so by~\eqref{eq:child-parent}
we know that $v\ne g(w)_k$. Hence $v \sim_Z g(w)_k$.		
	\end{proof}

	\begin{sublem}\label{lift2}		
		Let $v \in V[n_0]_k$. Let $\gamma=w_1\sim_Z w_2\sim_Z\cdots\sim_Z w_m$ be a path in $\Gamma[n_0]_1(B_v)$. Then there exists an index $n_1$ such that $w_1\sim_Z w_2\sim_Z\cdots\sim_Z w_{n_1}$ belongs to $\Gamma[n_0]_1(B_v)$, $v \sim_Z g(w_{n_1})_k$, and
\[
		\ell(v \sim_Z g(w_{n_1})_k) \leq \ell(w_1\sim_Z w_2\sim_Z\cdots\sim_Z w_{n_1}).
\]
\end{sublem}

		\begin{proof}
Let $n_1$ be the smallest index such that $w_1\sim_Z w_2\sim_Z\cdots\sim_Z w_{n_1}$ belongs to $\Gamma[n_0]_1(B_v)$. This is well-defined. 
From the previous sublemma, $v \sim_Z g(w_{n_1})_k$. Lemma~\ref{lem:2.10} gives the last bound.
		\end{proof}
			
	By applying this lemma successively, we obtain:
	
	\begin{lem}\label{lift3}		
		Let $v \in V[n_0]_k$. Let $\gamma=w_1\sim_Z w_2\sim_Z\cdots\sim_Z w_m$ be a path in $\Gamma[n_0]_1(B_v)$. Then there exists $N$ and indices $0<n_1<n_2<\ldots<n_N=m$ such that, if we set $v_j={g(w_{n_j})}_k$, for each $1\leq j\leq N$, then $w_{n_j}\sim_Z \cdots \sim_Z w_{n_{j+1}}$ belongs to $\Gamma[n_0]_1(B_{v_j})$ and $v_j \sim_Z v_{j+1}$. Moreover, with $v_0:=v$, we have
\[
		\ell(v_0 \sim_Z v_1 \sim_Z \ldots \sim_Z v_N) \leq \ell(w_1\sim_Z w_2\sim_Z\cdots\sim_Z w_{n_N}).
\]
	\end{lem}
	
	\begin{proof}
		We know from the hypothesis that there is at least one $j\in\{1,\cdots, m\}$ such that $2B_v$ intersects $B_{w_{j-1}}$
		and $B_{w_j}\setminus 2B_v$ is nonempty; note that by Sublemma~\ref{lift2}, $\ell(w_1\sim_Z\cdots\sim_Z w_j)\ge \ell(v\sim_Z g(w_j)_k)$.
		We denote by $n_1$ 
the largest such $j$, and set $v_1:=g(w_{n_1})_k$. If $n_1=m$, then we are
		done with the proof. So, assume $n_1<m$. 
		
		If $w_{n_1}\sim_Z\cdots \sim_Z w_m$ is a path in $\Gamma[n_0]_1(B_{v_1})$, then we apply Sublemma~\ref{lift2} and the above 
		procedure to this remaining path to choose the largest $j\in\{n_1,\cdots, m\}$ for which
		$v_1\sim_Z g(w_j)_k$ and $\ell(w_{n_1}\sim_Z\cdots\sim_Z w_j)\ge \ell(v_1\sim_Z g(w_j)_k)$. Set $n_2$ to be this largest index $j$, and $v_2:=g(w_{n_2})_k$.
				
		We continue this inductively to choose $n_1,n_2,\cdots, n_N$ and 
		$v_1=g(w_{n_1})_k,v_2=g(w_{n_2})_k,\cdots, v_N=g(w_{n_N})_k$, until either $n_N=m$, at which point we have completed the proof,
		or else we reach the point where the remaining segment $w_{n_N}\sim_Z\cdots\sim_Z w_m$ is not in
		the family $\Gamma[n_0]_1(B_{v_N})$. In this case, we see that $g(w_m)_k\sim_Z v_{N}$ or else $g(w_m)_k=v_N$.
		If $g(w_m)_k=v_N$ again we are done. 
		
		 We now focus on this remaining case where $w_{n_N}\sim_Z\cdots\sim_Z w_m$ is not in $\Gamma[n_0]_1(B_{v_N})$
			but at the same time $w(w_m)_k\ne v_N$. 
		Then from~\eqref{n0-condition},
		\begin{align*}
		d(x_{v_{N-1}},&x_{g(w_m)_k})\le d(x_{v_{N-1}}, x_{w_{n_N}})+d(x_{w_{n_N}},x_{v_N})+d(x_{v_N},x_{w_m})+d(x_{w_m},x_{g(w_m)_k})\\
		 <&[2\alpha^{-n_0k}+2\tau\alpha^{-n_0(k+1)}]+[\alpha^{-n_0k}+\alpha^{-n_0(k+1)}]+[2\alpha^{-n_0k}]+[\alpha^{-n_0k}+\alpha^{-n_0(k+1)}]\\
		 =&\alpha^{-n_0k}[6+2(1+\tau)\alpha^{-n_0}]<\tau\alpha^{-n_0k}.
		 \end{align*}
		 It follows that $v_{N-1}\sim_Z g(w_m)_k$; in this case we replace $v_{N-1}\sim_Z v_N$ with the path $v_{N-1}\sim_Z g(w_m)_k$,
		 and note that 
		 \[
		 \ell(v_{N-1}\sim_Zg(w_m)_k)\le \ell(w_{n_{N-1}}\sim_Z\cdots\sim_Z w_{n_{N}}\sim_Z\cdots\sim_Z w_m).
		 \]
		 Now we redefine $v_N$ to be $g(w_m)_k$ to complete the proof.
	\end{proof}
	
	\begin{sublem}\label{lift4}
		Suppose $\gamma=w_0\sim_Z w_1\sim_Z\cdots\sim_Z w_m$ is a path from $v=w_0 \in V[n_0]_{k}$ to $v'=w_m \in V[n_0]_{k}$, 
		such that $w_j\in V[n_0]_{k+1}$ when $1\le j\le m-1$. 
		Then there exists a path $\gamma'=v \sim_Z v_1 \sim_Z \ldots \sim_Zv_M=v'$ whose vertices are all in $V[n_0]_{k}$ and $\ell(\gamma')\le \ell(\gamma)$.	
	\end{sublem}
	
	\begin{proof}
		Necessarily, $v\sim_Z w_1$ and $v' \sim_Z w_{m-1}$ are vertical edges. If $v \sim_Z v'$, then $\gamma'=v \sim_Z v'$ will satisfy the claim, due to \eqref{eq:vert-horiz}, completing the proof.  So, assume that $v$ and $v'$ are not neighbors.
		In this case, necessarily the path $w_1\sim_Z\cdots\sim_Z w_{m-1}$ is in $\Gamma[n_0]_1(B_v)$. 
		We apply Lemma~\ref{lift3} and find $n^*\leq m$, $N$ and $v_j \in V[n_0]_k, 1\leq j \leq N$ such that the path $\gamma''$ given by $v=g(w_1)_k \sim_Z v_1 \sim_Z \ldots \sim_Z v_{N}=g(w_{n^*})$ satisfies
		\begin{equation*}\label{lifteq1}
			\ell(\gamma'') \leq \ell(w_1 \sim_Z \ldots \sim_Z \sim_Z w_{n^*}),
		\end{equation*}
		and that $v_{N} \sim_Z v'$ or $v_{N}=v'$. In the first case, we set $M=N+1$ and we choose $v_M=v'$, and in the second case we set $M=N$.
		In the second case the conclusion of the lemma is immediate with $\gamma':=\gamma''$. In the first case, note that $v_N\sim_Z v_{M}$, and
		 by~\eqref{eq:vert-horiz}, $\ell(v_{N} \sim_Z v_M) \leq \ell(w_{n^*} \sim_Z \ldots \sim_Z w_m) $.
		Therefore, concatenating $\gamma''$ with the edge $v_{N} \sim_Z v_M$ gives us the desired $\gamma'$.
		
	\end{proof}
	
	\begin{sublem}\label{lift5}
		Suppose $\gamma=w_0\sim_Z w_2\sim_Z\cdots\sim_Z w_m$ is a path from $v=w_0 \in V[n_0]_{k}$ to $v'=w_m \in V[n_0]_{k}$ whose vertices are all in $\bigcup_{n=0}^{k+1}V[n_0]_n$. Then there exists a path $\gamma'$ from $v$ to $v'$ whose vertices all lie in $\bigcup_{n=0}^{k}V[n_0]_n$ and $\ell(\gamma') \leq \ell(\gamma)$.	
	\end{sublem}
	
	\begin{proof}		
	If $\gamma$ already lies in $\bigcup_{n=0}^{k}V[n_0]_n$, then $\gamma'=\gamma$ ends the proof. Otherwise, there exists an integer $q$ with
	$ 1 \le q \le m$ and two sequences $n_j$ and $m_j$ of indices such that
\begin{itemize}
\item  $0\le n_1 < m_1 \leq n_2 < m_2 \le \cdots \leq n_q < m_q \le m$ (with merely $0\le n_1<m_1\le m$ if $q=1$),
\item  for all $1\leq j\le q$ the subpath $\gamma_j:=w_{n_j} \sim_Z \ldots \sim_Z w_{m_j}$ satisfies the assumptions of Sublemma~\ref{lift4}, and
\item the ``complement'' of the union of $\gamma_j$ consists of a finite (possibly empty) number of subpaths $\tilde{\gamma}_i$ of $\gamma$ that have all of their vertices in $\bigcup_{n=0}^{k}V[n_0]_n$ and that each $\tilde{\gamma}_i$ has its endpoints among $\{w_0,w_{n_1},w_{m_1},\cdots,w_{n_q}, w_{m_q},w_m\}$.
\end{itemize}
For each $j$, we apply  Sublemma~\ref{lift4} and replace $\gamma_j$ with a \emph{shorter} path $\gamma_j'$ that has all of its vertices in $V[n_0]_k$ and has the same endpoints as $\gamma$.
				
Clearly, the collection of $\gamma_j', 1\le j \le q$ concatenate with the collection of $\tilde{\gamma}_i$, described above, to form a path $\gamma'$ that satisfies the claim of the lemma.
	\end{proof}
	
	\begin{lem}\label{lift6}
		If $\gamma$ is a path in $Z[n_0]$ between two distinct vertices $v,v'\in V[n_0]_k$, then there exists a path $\gamma'$ between $v$ and $v'$
		so that $\ell(\gamma')\le \ell(\gamma)$ and all the vertices of $\gamma'$ lie in $\bigcup_{n=0}^kV[n_0]_n$
	\end{lem}
	
	\begin{proof}
		By repeatedly (but only finitely many times) applying Sublemma~\ref{lift5} to subcurves whose vertices lie in $\bigcup_{n=k}^\infty V[n_0]_n$,
		the proof is complete.
	\end{proof}

Lemma~\ref{lift6} shows us that to control the lengths of paths connecting $v,v'\in V[n_0]_k$ from below, 
it suffices to consider only the paths
connecting $v$ and $v'$ but staying within level $k$, that is, all the vertices of the path lie in 
$\bigcup_{n=0}^kV[n_0]_n$. We next
show that for such a path, there is a shorter path with vertices in $\bigcup_{n=0}^{k-1}V[n_0]_n$ 
connecting $g(v)_{k-1}$ to $g(v')_{k-1}$.
Since $v,v'$ do not satisfy the conditions listed in Remark~\ref{rem:simples}, we know that $k\ge 2$.

\begin{lem}\label{lift7} 
Suppose $\gamma$ is a path in $Z[n_0]$ between two vertices $v,v'\in V[n_0]_k$  such that all the vertices of 
$\gamma$ lie in $\bigcup_{n=0}^kV[n_0]_n$. Then there exists a path $\gamma'$ from $g(v)_{k-1}$ to $g(v')_{k-1}$  
such that all the vertices of $\gamma$ lie in $\bigcup_{n=0}^{k-1}V[n_0]_n$, and $\ell(\gamma') \leq \ell(\gamma)$.
\end{lem}

\begin{proof}
By Remark~\ref{rem:simples}, we know that $g(v)_{k-1}\neq g(v')_{k-1}$.
We write $\gamma$ as $v=v_0 \sim_Z v_1 \sim_Z \ldots \sim_Z v_m=v'$. If $\gamma$ completely lies in $V[n_0]_k$, then by Sublemma~\ref{lem:short} (in the case that the curve is not in $\Gamma[n_0]_1(B_{v})$), or
Lemma~\ref{lift3} (in the case that the curve is in $\Gamma[n_0]_1(B_{v})$), as $g(v)_{k-1}\not\sim_Z g(v')_{k-1}$ because of Remark~\ref{rem:simples},
there must exist a path  giving us the desired $\gamma'$.

So, we continue by assuming that $\gamma$ does not lie in $V[n_0]_k$. Let $n_1$ be the largest integer that satisfies: $v_{j} \in V[n_0]_k$, for all $0\leq j <n_1$. Necessarily, $v_{n_1}=g(v_{n_1-1})_{k-1}$. By Sublemma~\ref{lem:short}, together with~\eqref{eq:vert-horiz}, or by
Lemma~\ref{lift3}, either $g(v_0)_{k-1}=v_{n_1}$, or we can find a path from $g(v_0)_{k-1}$ to $v_{n_1}$  
that lies in $V[n_0]_{k-1}$ and is shorter (in $\ell$ distance) than $v_0 \sim_Z v_1 \sim_Z \ldots \sim_Z v_{n_1}$. By analogous argument at the other end of $\gamma$ terminating at $v'$, there exists an intger $n_2\geq n_1$ such that  $v_{j} \in V[n_0]_k$, for all $n_2 < j \leq m$,  $v_{n_2}=g(v_{n_2+1})_{k-1}$, and either $g(v')_{k-1}=v_{n_2}$, or we can find a path from $g(v')_{k-1}$ to $v_{n_2}$ that lies in $V[n_0]_{k-1}$ and is shorter (in $\ell$ distance) than $v_{n_2} \sim_Z \ldots \sim_Z v'$.

If needed, we apply Lemma~\ref{lift6} to replace $v_{n_1}\sim_Z \ldots \sim_Z v_{n_2}$ with a path that lies in  $\bigcup_{n=0}^{k-1}V[n_0]_n$ and is shorter.

Concatenation of the three replacement curves (some possibly degenerate, i.e.\  consisting of just a vertex) gives us the desired $\gamma'$.
\end{proof}

\begin{prop}\label{prop:SameLevel-H3}
Suppose that $v,v'\in V[n_0]_k$ satisfy the condition of Remark~\ref{rem:simples}. Then for each curve $\gamma$ in $Z[n_0]$ with end points
$v,v'$, we have that $\ell(\gamma)\ge K_0^{-3} \pi(z_{v,v'})$.
\end{prop}

\begin{proof}
	An application of Lemma~\ref{lift7} gives a curve $\gamma_1$ with end points $g(v)_{k-1}$, $g(v')_{k-1}$ and $\ell(\gamma)\ge \ell(\gamma_1)$.
By Remark~\ref{rem:simples} we know that $g(v)_{k-1}\ne g(v')_{k-1}$ and that $g(v)_{k-1}\not\sim_Zg(v')_{k-1}$. Thus we can 
apply Lemma~\ref{lift7} to $\gamma_1$ to obtain a curve $\gamma_2$ with endpoints $g(v)_{k-2}$ and $g(v')_{k-2}$. 

If $g(v)_{k-2}=g(v')_{k-2}$ or $g(v)_{k-2}\sim_Z g(v')_{k-2}$, then $z_{v,v'}\in\{g(v)_{k-2}, g(v')_{k-2}\}$, and the 
claim of the proposition holds. If this assumption does not hold, then we can apply Lemma~\ref{lift7} to $\gamma_2$ to obtain 
a curve $\gamma_3$ with end points $g(v)_{k-3}$ and $g(v')_{k-3}$.

We proceed inductively to construct paths $\gamma_j$, $j=1,2,\cdots, m-1$ 
with end points $g(v)_{k-j}$ and $g(v')_{k-j}$ and $\ell(\gamma_j)\le \ell(\gamma)$, such that all the vertices of $\gamma_j$
lie in $\bigcup_{n=0}^{k-j}V[n_0]_n$, 
until we encounter the situation where $g(v)_{k-m}=g(v')_{k-m}$
or $g(v)_{k-m}\sim_Z g(v')_{k-m}$. This inductive process does terminate eventually, and $z_{v,v'}\in\{g(v)_{k-m}, g(v')_{k-m}\}$.

Note that we have $\ell(\gamma)\ge \ell(\gamma_m)$. By the construction of the length function $\ell$, and by 
the condition~(H2) verified earlier for $\pi$, 
we see that
\begin{align*}
\ell(\gamma_m)\ge \frac{1}{K_0^2}\left[\pi(g(v)_{k-m+1})+\pi(g(v')_{k-m+1})\right]
&\ge \frac{1}{K_0^3}\, \left[\pi(g(v)_{k-m})+\pi(g(v')_{k-m})\right] \\
&\ge \frac{1}{K_0^3}\, \max\{\pi(g(v)_{k-m}),\, \pi(g(v')_{k-m})\}\\
&\ge \frac{1}{K_0^3}\, \pi(z_{v,v'}).
\end{align*}
This completes the proof of the proposition; recall that $K_0$ is from~\eqref{eq:K0}.
\end{proof}

Having considered the situation where the end points of $\gamma$ lie in the same level $V[n_0]_k$, we now turn our attention to 
the situation where the end points of $\gamma$ lie in two different levels. Recall that we continue the assumption outlined in
Remark~\ref{rem:simples}. To remind ourselves that the two vertices are of different levels, we denote them by $v,w$ rather than $v,v'$.

\begin{sublem}\label{lift8}
Suppose $\gamma$ is a path in $Z[n_0]$ between two vertices $v,w\in V[n_0]$ with $k=n_v > n_{w}$. Then  
there exists a path $\gamma'$ from $g(v)_{k-1}$ to $w$ such that all the vertices of $\gamma'$ lie in $\bigcup_{n=0}^{k-1}V[n_0]_n$ and $\ell(\gamma')\leq \ell(\gamma)$.
\end{sublem}

\begin{proof}
Write  $\gamma$ as $v=v_0 \sim_Z v_1 \sim_Z \ldots \sim_Z v_m=w$ and recall 
from Remark~\ref{rem:simples} that $g(v)_{k-1}\not\sim_Z w$ and $g(v)_{k-1}\neq w$. 
Let $N$ be the smallest 
integer that satisfies: $n_{v_j} < k$ for all $N < j \leq m$. Necessarily, $v_{N+1}=g(v_N)_{k-1}$. 
An application of Lemma~\ref{lift6} to the segment $v_0\sim_Z\cdots v_N$ allows us to replace this segment with a shorter path
that lies entirely in $\bigcup_{n=1}^k V[n_0]_n$. Thus, for the remainder of the proof, we will assume that $\gamma$ lies entirely in
$\bigcup_{n=1}^k V[n_0]_n$.
We now apply Lemma~\ref{lift7} to conclude that either $g(v)_{k-1}\sim_Z v_{N+1}$ or $g(v)_{k-1}=v_{N+1}$,
or else we can
find a path $\gamma''$ from $g(v)_{k-1}$ to $v_{N+1}$ that lies in $\bigcup_{n=0}^{k-1}V[n_0]_n$ and is shorter than $v=v_0 \sim_Z v_1 \sim_Z \ldots \sim_Z v_N$. 
If $g(v)_{k-1}=v_{N+1}$, then we can set $\gamma'$ to be the path $v_{N+1}  \sim_Z \ldots \sim_Z v_m=w$. If $g(v)_{k-1}\sim_Z v_{N+1}$, then
we can set $\gamma'$ to be the concatenation of $v_{N+1}  \sim_Z \ldots \sim_Z v_m=w$ with the edge-path $g(v)_{k-1}\sim_Z v_{N+1}$; in
this case, \eqref{eq:vert-horiz} together with the fact that the edge $v_N\sim_Z v_{N+1}$ is part of $\gamma$, tells us that $\ell(\gamma')<\ell(\gamma)$. If neither of these two options occur, then by Lemma~\ref{lift7} we know that the length (in $\ell$)
of the concatenation of $\gamma''$ and $v_{N+1}  \sim_Z \ldots \sim_Z v_m=w$ is shorter than that of $\gamma$.
\end{proof}

\begin{lem}\label{lift9}
Suppose $\gamma$ is a path in $Z[n_0]$ between vertices $v,w\in V[n_0]$ with $k=n_v > n_{w}=k'$. Then either  
$g(v)_{k'}=w$, or $g(v)_{k'}\sim_Z w$ (in which case $k'\le k-2$ because of Remark~\ref{rem:simples}), 
or there exists a path $\gamma'$ from $g(v)_{k'}$ to $w$ such that all the vertices of $\gamma'$ lie in $\bigcup_{n=0}^{k'}V[n_0]_n$ and $\ell(\gamma')\leq \ell(\gamma)$.
\end{lem}

\begin{proof}
We begin by applying Sublemma~\ref{lift8} to $\gamma$ to see that there exists a path $\gamma_1$ from $g(v)_{k-1}$ to $w$ that lies in $\bigcup_{n=0}^{k-1}V[n_0]_n$ with $\ell(\gamma_1)\leq \ell(\gamma)$. 

If $k-1=k'$, then we set $\gamma'=\gamma_1$ and it completes the proof. So, we assume $k-2\geq k'$. 
In this case, we may have $g(v)_{k'}=w$, or $g(v)_{k'}\sim_Z w$, in which case the claim in the lemma holds. 
Otherwise, we apply Sublemma~\ref{lift8} to the path $\gamma_1$ (which satisfies the assumptions with $g(v)_{k-1}$ in role of $v$) and conclude that either $g(v)_{k-2}=w$, or $g(v)_{k-2}\sim_Z w$, or there exists a path $\gamma_2$ from $g(v)_{k-2}$ to $w$ that lies in  $\bigcup_{n=0}^{k-2}V[n_0]_n$ with $\ell(\gamma_2)\leq \ell(\gamma_1)$. If the first two cases hold, the claim of the lemma holds. In the latter case, if $k-2=k'$, we set $\gamma'=\gamma_2$ and the proof is again complete; since $\ell(\gamma') \leq \ell(\gamma_1) \leq \ell(\gamma)$.

We continue inductively until (in finite time) we end up with $k-M=k'+1$: either $g(v)_{k'+1} \sim_Z w$, or a path $\gamma_M$ from $g(v)_{k'+1}$ to $w$ that lies in $\bigcup_{n=0}^{k'+1}V[n_0]_n$ with $\ell(\gamma_M)\leq \ell(\gamma)$. Now, one last application of Sublemma~\ref{lift8} to $\gamma_M$ gives: either $g(v)_{k'} = w$, or $g(v)_{k'}\sim_Z w$, or there exists a path $\gamma_{M+1}$ from $g(v)_{k'}$ to $w$ that lies in $\bigcup_{n=0}^{k'}V[n_0]_n$ with $\ell(\gamma_{M+1})\leq \ell(\gamma_M)$. The claim of the lemma clearly holds in the first two cases. If the latter case occurs, we set $\gamma'=\gamma_{M+1}$ and it satisfies the claim of the lemma thanks to $\ell(\gamma_{M+1}) \leq \ell(\gamma_M) \leq \ell(\gamma)$.
\end{proof}

\begin{prop}\label{prop:not-same-level}
Suppose that $v\in V[n_0]_k$ and $w\in V[n_0]_{k'}$ with $k'\le k-1$, does not satisfy any of the conditions listed in Remark~\ref{rem:simples}.
Let $\gamma$ be a path in $Z[n_0]$ with end points $v,w$. Then
$\ell(\gamma)\ge K_0^{-3}\, \pi(z_{v,w})$.
\end{prop}
\begin{proof}
	If either $g(v)_{k'}=w$ or $g(v)_{k'}\sim_Z w$, then we have that $z_{v,w}\in\{g(v)_{k'},w\}$, and as $\gamma$ does end at $w$,
	necessarily we must have that 
	\[
	\ell(\gamma)\ge \frac{1}{K_0}\, \pi(w)\ge \frac{1}{K_0^2}\max\{\pi(w),\pi(g(v)_{k'})\}\ge \frac{1}{K_0^2}\pi(z_{v,w}).
	\]
	Recall the definition of $K_0$ from~\eqref{eq:K0} and the remark right after the definition of $K_0$ which pointed out 
	the comparison of the $\pi$-values of vertices that are \emph{horizontal} neighbors.
	For the remainder of this proof we will assume that neither of the above cases occurs. Then
	by Lemma~\ref{lift9} above, there is a curve $\gamma'$, with vertices in $\bigcup_{n=0}^{k'}V[n_0]_n$ and end points $w$,
	$g(v)_{k'}$, such that $\ell(\gamma')\le \ell(\gamma)$.
	
	If $g(v)_{k'}, w$ do satisfy any of the conditions listed in Remark~\ref{rem:simples}, then either $g(v)_{k'-1}=g(w)_{k'-1}$ or 
	$g(v)_{k'-1}\sim_Zg(w)_{k'-1}$, in which case we have that 
	\[
	\min\{\pi(g(v)_{k'-1},\pi(g(w)_{k'-1})\}\ge K_0^{-1}\max \{\pi(g(v)_{k'-1},\pi(g(w)_{k'-1})\},
	\]
	and so
	\[
	\ell(\gamma)\ge \ell(\gamma')\frac{1}{K_0}\, \pi(w)\ge\frac{1}{K_0^3}\max\{\pi(g(v)_{k'-1},g(w)_{k'-1})\}\ge \frac{1}{K_0^3}\, \pi(z_{v,w}).
	\]
If $g(v)_{k'}, w$ do not satisfy any of the conditions of Remark~\ref{rem:simples}, then we can apply 
Proposition~\ref{prop:SameLevel-H3} to $\gamma'$ and see that 
\[
\ell(\gamma)\ge \ell(\gamma')\ge \frac{1}{K_0^3}\, \pi(z_{g(v)_{k'},w}).
\]
Finally, note that $z_{v,w}=z_{g(v)_{k'},w}$ because  $g(v)_{k'}, w$ do satisfy the conditions of Remark~\ref{rem:simples} and by construction.
This completes the proof of the proposition.
\end{proof}

\begin{cor}\label{cor:h3-on-z}
 Condition (H3) holds on $Z[n_0]$. Namely, for any path $\gamma$ in $Z[n_0]$ with end points $v,w$, we have $\ell(\gamma)\ge K_0^{-6}\, \pi(z_{v,w})$.
\end{cor}
\begin{proof}
Let $v,w \in V[n_0]$ be arbitrary  and let $\gamma$ be a path in  $Z[n_0]$ that connects them. If $v$ and $w$ satisfy one of the exceptional cases of Remark~\ref{rem:simples}, then either $\gamma$ contains $z_{v,w}$, or it contains a vertex that is a neighbor of $z_{v,w}$. In either case, it follows that $\int_\gamma \pi \, ds \geq \frac{1}{K_0^2}\pi(z_{v,w})$.

So, in the rest of proof, assume $v$ and $w$ satisfy the condition of Remark~\ref{rem:simples}. The by Propositions~\ref{prop:SameLevel-H3} and \ref{prop:not-same-level}, we know that $\ell(\gamma)\ge K_0^{-3}\, \pi(z_{v,w})$.

Hence, the proof will be complete once we prove that for any path $\gamma$,
\begin{equation}\label{ell-to-pi}
{K_0^{-3}}\ell(\gamma) \leq \int_\gamma \pi \, ds.
\end{equation}
Toward verifying \eqref{ell-to-pi}, it is easy to see 
from the definition of $\ell$ given in~\eqref{eq:def-ell} 
that for horizontal edges $e=v\sim_Z v'$, we have $\ell(e)\le\min\{\pi(v),\pi(v')\} \leq \int_e \pi \,ds$. 
For vertical edges  $e=v\sim_Z v'$, without loss of generality we assume that $n_{v'}=n_v+1$; then by~\eqref{eq:def-ell} and condition (H2), 
\[
\ell(e) =K_0^2\pi^*(v')\le K_0^3\, [\pi(v)+\pi(v')]/2= K_0^3\int_e \pi \,ds.
\]
By adding these inequalities over all edges of $\gamma$, we obtain \eqref{ell-to-pi}. Thus, the proof of corollary is complete.
\end{proof}

\begin{thm}
 Condition (H3) holds on $G[n_0]$.  Namely, for any path $\gamma$ in $G[n_0]$ with end points $v,w$, we have 
 \[
 \int_\gamma\pi\, ds\ge \frac{1}{K_0^6\,(K_0+1)}\, \pi(z_{v,w}).
 \]
\end{thm}

\begin{proof}
	Let $\gamma$ be a path in $G[n_0]$ with end points $v,w$. 
	We will construct a path $\gamma'$ in $Z[n_0]$ with end points $v,w$ as follows so that 
	$\int_{\gamma'}\pi\, ds\le (K_0+1)\, \int_\gamma\pi\, ds$, and then an application of Corollary~\ref{cor:h3-on-z} yields the desired result.
	
	Let $e=v_1\sim v_2$ be an edge in $\gamma$. If $e$ is a horizontal edge, then it is
	also a horizontal edge in $Z[n_0]$, and we use that in $\gamma'$. If it is a vertical edge that also is in $Z[n_0]$, then again
	we keep that in $\gamma'$. If $e$ is a vertical edge that is not in $Z[n_0]$, then wihtout loss of generality we will consider the
	case that $n_{v_2}=n_{v_1}+1$. Then $g(v_2)_{n_{v_1}}\sim_Z v_1$, and we replace the edge $e$ in $\gamma$ by the short path
	$\beta(e)=v_1\sim_Z g(v_2)_{n_{v_1}}\sim_Z v_2$ in constructing $\gamma'$. Observe that
	\begin{align*}
	\int_{\beta(e)}\, \pi\, ds=\frac{\pi(v_1)+\pi(g(v_2)_{n_{v_1}})+\pi(g(v_2)_{n_{v_1}})+\pi(v_2)}{2}
	   &\le (K_0+1)\, \frac{\pi(v_1)+\pi(v_2)}{2}\\
	   &= (K_0+1)\, \int_e\pi\, ds.
	\end{align*}
	Doing these replacements one edge at a time in the path $\gamma$, we obtain $\gamma'$ as desired, with
	$\int_{\gamma'}\pi\, ds\le (K_0+1)\, \int_\gamma\pi\, ds$.
\end{proof}

\subsection{Conclusion}
Assuming $\Mod_p(X,d)=0$, we succeeded in finding a weight $\rho$ on the hyperbolic filling $G[n_0]$ of $X$ that,  
together with its associated $\pi$, satisfies conditions (H1) through (H4) of Theorem~\ref{thm:Car1.1}. 
Therefore, by Theorem~\ref{thm:Car1.1}, $d_\rho$  is in the Ahlfors regular quasisymmetric 
	gauge of $X$ and $(X,d_\rho)$ is Ahlfors $p$-regular. Thus, the Ahlfors regular conformal dimension 
	of $X$ is less than or equal to $p$.


\begin{thebibliography}{A}
\frenchspacing
\bibitem{BBS} A. Bj\"orn, J. Bj\"orn, N. Shanmugalingam:
\emph{Extension and trace results for doubling metric measure spaces and their hyperbolic fillings.}
 J. Math. Pures Appl. (9) {\bf 159} (2022), 196--249.
 \bibitem{BHK} M. Bonk, J. Heinonen, P. Koskela:
 \emph{Uniformizing Gromov hyperbolic spaces.}
  Ast\'erisque {\bf 270} (2001), viii+99 pp. 
 \bibitem{BoSa} M. Bonk, E. Saksman:
 \emph{Sobolev spaces and hyperbolic fillings.} 
 J. Reine Angew. Math. {\bf 737} (2018), 161--187.
 \bibitem{BoSc} M. Bonk, O. Schramm:
 \emph{Embeddings of Gromov hyperbolic spaces.} Geom. Funct. Anal. {\bf 10} (2000), no. 2, 266--306.
 \bibitem{BK} M. Bourdon, B. Kleiner:
 \emph{Some applications of $\ell_p$-cohomology to boundaries of Gromov hyperbolic spaces.}
 	Groups Geom. Dyn. {\bf 9} (2015), no. 2, 435--478.
 \bibitem{BP} M. Bourdon, H. Pajot:
 \emph{Cohomologie $\ell_p$ et espaces de Besov.} 
J. Reine Angew. Math. {\bf 558} (2003), 85--108.
\bibitem{BuSch} S. Buyalo, V. Schroeder:
\emph{Elements of asymptotic geometry.}
EMS Monographs in Mathematics. European Mathematical Society (EMS), Z\"urich, 2007. 
\bibitem{Car} M. Carrasco Piaggio:
\emph{On the conformal gauge of a compact metric space.}
 Ann. Sci. \'Ec. Norm. Sup\'er. (4) {\bf 46} (2013), no. 3, 495--548.
 \bibitem{DS}G. David, S. Semmes: 
 \emph{Fractured fractals and broken dreams. Self-similar geometry through metric and measure.} 
 Oxford Lecture Series in Mathematics and its Applications {\bf 7}. The Clarendon Press, 
 Oxford University Press, New York, 1997. x+212 pp.
\bibitem{Gro} M. Gromov:
\emph{Hyperbolic groups.} 
Essays in group theory, 75--263,
Math. Sci. Res. Inst. Publ., 8, Springer, New York, 1987.
\bibitem{Hei} J. Heinonen:
\emph{Lecture notes on analysis in metric spaces.}
Springer Universitext, Springer Verlag New York (2001).
\bibitem{Kwa} J. Kwapisz:
\emph{Conformal dimension via p-resistance: Sierpiński carpet.}
Ann. Acad. Sci. Fenn. Math. {\bf 45} (2020), no. 1, 3--51.
\bibitem{Mur} M. Murugan:
\emph{Conformal Assouad dimension as the critical exponent for combinatorial modulus.}
Ann. Fenn. Math. {\bf 48} (2023), no. 2, 453--491.
\bibitem{Sh} N. Shanmugalingam:
\emph{On Carrasco Piaggio's theorem characterizing quasisymmetric maps from compact doubling spaces to Ahlfors regular spaces.}
Potentials and partial differential equations--the legacy of David R. Adams, Adv. Anal. Geom., {\bf 8} 23--48, 
De Gruyter, Berlin, (2023).
\end{thebibliography}
\end{document}